%% file: sn-article.tex
\documentclass[sn-mathphys-num]{sn-jnl}
\usepackage{graphicx}%
\usepackage{multirow}%
\usepackage{amsmath,amssymb,amsfonts}%
\usepackage{amsthm}%
\usepackage{mathrsfs}%
\usepackage[title]{appendix}%
\usepackage{xcolor}%
\usepackage{textcomp}%
\usepackage{manyfoot}%
\usepackage{booktabs}%
\usepackage{algorithm}%
\usepackage{algorithmicx}%
\usepackage{algpseudocode}%
\usepackage{listings}%
\usepackage{array}
\usepackage{longtable}
\usepackage{todonotes}
\usepackage{diagbox}

\theoremstyle{thmstyleone}%

\newtheorem{theorem}{Theorem}

\theoremstyle{thmstyletwo}%
\newtheorem{example}{Example}%

\theoremstyle{thmstylethree}%
\newtheorem{definition}{Definition}%
\newtheorem{lemma}{Lemma}
\newtheorem{corollary}{Corollary}
\newtheorem{observation}{Observation}

\newcommand{\conv}{\hbox{conv}}

\begin{document}

\title[Efficient exact sequential lifting algorithm for binary knapsack set]
{Efficient exact sequential lifting algorithm for binary knapsack set}

\author[1]{\fnm{Xintong} \sur{Wang}}\email{xintong\_wang@bit.edu.cn}

\author*[2]{\fnm{Liang} \sur{Chen}}\email{chenliang@lsec.cc.ac.cn}

\author[2]{\fnm{Yu-Hong} \sur{Dai}}\email{dyh@lsec.cc.ac.cn}

\affil*[1]{\orgdiv{School of Mathematics and Statistics}, \orgname{Beijing Institute of Technology}, \orgaddress{\city{Beijing}, \postcode{100081}, \country{China}}}

\affil[2]{\orgdiv{Institute of Computational Mathematics and Scientific/Engineering Computing}, \orgname{Academy of Mathematics and Systems Science, Chinese Academy of Sciences}, \orgaddress{\city{Beijing}, \postcode{100190}, \country{China}}}


\abstract{
Lifting is a crucial technique in mixed integer programming (MIP) for generating strong valid inequalities, which serve as cutting planes to improve the branch-and-cut algorithm.
We first propose an exact sequential lifting algorithm for the binary knapsack set, which employs the dominance list structure to remove redundant storage and computation in the dynamic programming (DP) array.
This structure preserves scale invariance and effectively handles constraints with non-integer coefficients. 
Then, a reduction method is developed for the lifting procedure under some conditions, further enhancing computational efficiency.
Finally, numerical experiments demonstrate that the proposed algorithm outperforms DP with arrays in terms of both efficiency and stability, particularly for large-scale and large-capacity instances.
Moreover, it enables exact sequential lifting for binary knapsack sets with non-integer weights and large capacities, making it directly applicable in modern MIP solvers.
}



\maketitle
\input{tex/intro}

\input{tex/formulation}
\input{tex/method}

\input{tex/result}
\input{tex/appendix}
\bibliography{tex/sn-bibliography}

\end{document}

%% file: tex/intro.tex
\section{Introduction}
Consider the binary knapsack set
\begin{equation}\label{kp}
\mathcal{X} = \left\{\boldsymbol{x}\in \{0,1\}^n \,:\,\sum_{i=1}^n a_ix_i \le b\right\}
\end{equation}
where \(a_i \in \mathbb{Q}_+\) denotes the weight of item \(i\) for \(i \in N := \{1, 2, \ldots, n\}\), and \(b \in \mathbb{Q}_+\) is the knapsack capacity.
Without loss of generality, we assume $0< a_{i}\le b$ for all $ i\in N$ and $\sum_{i=1}^n a_{i} >b$.
The convex hull of \(\mathcal{X}\), denoted by
$\mathcal{P} := \conv(\mathcal{X})$, is referred to as the binary knapsack polytope.

Binary knapsack set of the form \eqref{kp} arises as feasible regions in many mixed integer programming (MIP) problems, including the generalized assignment problem \citep{ross1975branch}, the capacitated network location problem \citep{ceselli2009computational} and the multiple knapsack problem \citep{ferreira1996solving}.
Strong valid inequalities for binary knapsack polytope $\mathcal{P}$ \citep{balas1975facets, wolsey1975faces, crowder1983solving, weismantel19970} are widely used as cutting planes in the state-of-the-art MIP solvers and play an important role in tackling general MIP problems.
Empirical studies emphasize the critical importance: disabling methods such as knapsack separation significantly affect the performance of solvers \citep{achterberg2013mixed}. 
Such inequalities reduce the search space by eliminating infeasible solutions, thereby improving computational efficiency.
Therefore, deriving strong valid inequalities for \(\mathcal{X}\) is of both theoretical and practical significance.
However, obtaining them for high dimensional sets remains a challenging problem \citep{fawzi2022lifting}.
To address this challenge, \textit{lifting} is one of primary approaches for generating strong valid inequalities.

The lifting procedure begins with a valid inequality for a projection of the feasible region (obtained by fixing some variables), referred to as the \emph{seed inequality}, and terminates with a valid inequality for the original space by relaxing the fixed variables \citep{wolsey1976facets}.
More precisely, let 
\begin{equation}\label{c}
(C, N_0, N_1)
\end{equation} 
be a partition of the index set  \( {N} \).
Then the fixed knapsack set is defined as 
$\mathcal{X}(N_0, N_1) := \mathcal{X} \cap \left\{\boldsymbol{x} \, : \, x_i = 0, \, \forall\, i \in N_0; \, x_i = 1, \, \forall\, i \in N_1\right\},$ i.e.,
\begin{equation}\nonumber
\begin{aligned}	
\mathcal{X}(N_0, N_1) = \Bigg\{ \boldsymbol{x} \in \{0,1\}^n : &\sum_{i \in C} a_i x_i \le
b - \sum_{i \in N_1} a_i; \, x_i = 0, \, \forall\, i \in N_0; \\
&\, x_i = 1, \, \forall\, i \in N_1 \Bigg\}.
\end{aligned}
\end{equation}
Assuming that $C = \{1, 2, \ldots, \ell\},$
we start with a seed inequality of the form
\begin{equation}\label{seed}
\sum_{i\in C} \alpha_i x_i \le \beta^\ell,
\end{equation}
which is valid for the convex hull of the set \( \mathcal{X}(N_0, N_1) \). 
The procedure of lifting $x_i$ for $i\in N_0$ is referred to as \emph{up lifting}, 
while lifting $x_i$ for $i\in N_1$ is called \emph{down lifting}.
Note that both up and down lifting can be applied jointly to a single inequality \citep{wolsey1999integer}.
The final lifted inequality then takes the form 
$$\sum_{i\in C} \alpha_i x_i + \sum_{i \in N_0} \alpha_ix_i+\sum_{i \in N_1} \alpha_ix_i\le \beta^\ell + \sum_{i \in N_1} \alpha_i,$$ 
which is valid for original knapsack polytope $\mathcal{P}$.

\subsection{Literature review}
Lifting was first introduced by \citet{gomory1969some} in the context of the group problem.
In the mid-1970s, lifting gained significant attention when it was applied to structured MIP problems, such as the knapsack problem \citep{balas1975facets}, independence systems problem \citep{,hammer1975facet}, and vertex packing peoblem \citep{padberg1973facial, nemhauser1974properties}.
One important class of valid inequalities is the lifted cover inequality (LCI) \citep{balas1975facets}.
\citet{balas1978facets} provided a complete characterization of the facet-defining inequalities obtained from minimal covers, and subsequent studies proposed various generalizations of LCIs \citep{zemel1978lifting, wolsey1999integer, gu2000sequence, easton2008simultaneously, letchford2019lifted}.
Both recognizing facet-defining inequalities and computing general LCIs are known to be $\mathcal{NP}$-hard \citep{hartvigsen1992complexity, chen2021complexity}.
\citet{padberg1975note} identified lifting as a generic computational tool for generating strong inequalities for binary integer programs.
The method was later formalized and extended to unstructured integer programs \citep{wolsey1976facets, wolsey1977valid, zemel1978lifting, balas1978facets, crowder1983solving, gu1999lifted2, gu2000sequence};
and generalized to broader settings, including general integer variables  \citep{ceria1998cutting}, relaxation with continuous variables \citep{marchand19990, richard2002lifted, richard2003lifted, atamturk2004sequence};
as well as specialized global optimization problems such as the multidimensional knapsack problem \citep{kaparis2008local}, the infinite group problem \citep{richard2009valid, dey2010constrained, basu2012unique, basu2019nonunique};
and families of valid inequalities for various sets, including hypomatchable inequalities \citep{martin1999integer}, complementarity constraints \citep{de2002facets}, two-integer knapsack inequalities \citep{agra2007lifting}, lifted tableaux inequalities \citep{narisetty2011lifted}, disjoint cardinality constraints \citep{zeng2011polyhedral, zeng2011polyhedral2}, and mixing inequalities \citep{dey2010composite}. 
Furthermore, fewer studies have extended lifting techniques from MIP problems to linear programs \citep{hoffman1991improving, dey2010composite}, nonlinear programs \citep{richard2010lifting, gomez2018submodularity, nguyen2018deriving}, nonconvex quadratic programs \citep{vandenbussche2005polyhedral, lin2008box}, conic integer programming \citep{atamturk2011lifting}, and mixed integer bilinear programs \citep{gupte2012mixed, chung2014lifted, gu2023lifting}.

Lifting technique can be classified into sequential and simultaneous.
\textit{Sequential lifting} modifies one coefficient (and possibly the capacity) at a time, which played a pivotal role in strengthening branch-and-cut algorithms for binary integer programs, particularly knapsack problems \citep{wolsey1975faces, wolsey1976facets, crowder1983solving,  zemel1989easily, weismantel19970, gu1998lifted, kaparis2010separation, zeng2011polyhedral}. 
It has also been extensively used to derive facet-defining inequalities in other contexts \citep{padberg1973facial, nemhauser1974properties, balas1975facets, hammer1975facet, balas1978facets,  marchand19990, richard2002lifted, van2002polyhedral, richard2003lifted, atamturk2003facets}.
\textit{Simultaneous lifting}, or \textit{sequence independent lifting}, in contrast, ensures coefficients are independent of the lifting order and often relies on superadditive functions.
\citet{wolsey1977valid} established its connection with superadditive functions, 
while \citet{gu1994lifted} and \citet{atamturk2004sequence} extended it to mixed binary and mixed integer linear programs, respectively.
In this paper we focus on sequential lifting, for further developments in simultaneous lifting, see \citep{gu1998lifted, gu1999lifted, gu1999lifted2, gu2000sequence, atamturk2003facets, shebalov2006sequence, easton2008simultaneously, kubik2009simultaneously}.

In practice, lifting generally involves solving optimization problems, leading to a trade-off between computational efficiency and the strength of the resulting inequalities.
\emph{Exact lifting} computes optimal coefficients, preserving facet-defining properties but requiring substantial computational effort \citep{wolsey1975faces, zemel1978lifting, easton2008simultaneously}. 
If the seed inequalities are facets in a lower-dimensional space, exact lifting preserves this property in higher dimensions \citep{padberg1975note, wolsey1999integer}. 
\emph{Approximate lifting}, in contrast, generates valid but not necessarily facet-defining inequalities. 
Although theoretically weaker, it is computationally efficient and particularly suitable for large-scale MIP problems \citep{balas1975facets, gu2000sequence}. 
Today, approximate lifting method is standard in major MIP solvers, including HiGHS \citep{huangfu2018parallelizing}, CBC \citep{cbc} and SCIP \citep{bolusani2024scip}, due to its computational efficiency.
These implementations typically rely on superadditive functions for efficiency, though the resulting coefficients may be insufficiently strong to define facets. 
Consequently, while approximate lifting offers significant practical benefits, it sacrifices some theoretical strength, potentially limiting the effectiveness of branch-and-cut frameworks.

\subsection{Motivation and main contributions}
Binary knapsack set frequently arises in industrial applications, such as the generalized assignment problem \citep{savelsbergh1997branch}, scheduling problem \citep{potts1988algorithms}, and bandwidth packing problem \citep{han2013exact}.
However, these problems often involve non-integer weights or extremely large capacities, which make significant computational challenges for dynamic programming (DP) methods used for exact lifting.
Meanwhile, current MIP solvers typically rely on approximate lifting algorithms to solve large-scale practical problems.
The motivation of this work is to develop an exact sequential lifting algorithm for binary knapsack set to address these challenges and improve computational efficiency.

The main contributions of this paper are summarized as follows.
First, we propose an exact sequential lifting algorithm for the binary knapsack set that exploits dominance relationship through a dominance list structure.
This algorithm eliminates redundant storage during the lifting process and enables the effective handling of non-integer weights and large capacities. 
Moreover, the proposed algorithm exhibits scale invariance, ensuring stable computational performance when the weights and capacity are proportionally scaled.
Second, we introduce a reduction method applicable to both dominance lists and DP arrays under some conditions, which further improves computational efficiency. 
Finally, we present numerical experiments across various settings, which confirm the effectiveness and scale invariance of the proposed algorithm.
Moreover, it enables exact sequential lifting for binary knapsack sets with non-integer weights and large capacities, making it suitable for integration into modern MIP solvers.

The remainder of the paper is organized as follows. 
Section \ref{sec2} reviews related work on exact sequential lifting and DP with arrays for binary knapsack set.
Section \ref{sec3} introduces the concept of the dominance list, presents the reduction method, and proposes the exact sequential lifting algorithm.
Section \ref{numerial} presents the computational experiments.
Finally, Section \ref{summary} concludes the paper and outlines future research directions.

%% file: tex/formulation.tex
\section{Sequential lifting on binary knapsack set}\label{sec2}
Sequential lifting is a procedure in which the projected variables are lifted one by one according to a given specified order. 
The resulting inequality highly depends on this order, since different sequences may yield different lifting coefficients.
At each lifting step, the coefficient is determined by solving an individual optimization problem that provides the strongest valid value.
It is important to note that if the seed inequality \eqref{seed} defines the facet of \( \conv(\mathcal{X}(N_0, N_1)) \), the final inequality will define the facet of $\conv(\mathcal{X})$.

Let the lifting sequence be \( \{\ell+1, \ldots, n\} \), with $C = \{1,\ldots,\ell\},~N_0^\ell = N_0$ and $N_1^\ell = N_1$ (the partition of $N$ in \eqref{c}).
Suppose that the variables $x_{\ell+1}, \ldots, x_k$ have already been lifted for some $k>\ell$. 
The inequality
$$\sum_{i=1}^k \alpha_i x_i\le \beta^k$$
is valid for the set
\begin{equation}\nonumber
\begin{aligned}
\mathcal{X}(N_0^k,N_1^k)=\Bigg\{\boldsymbol{x} \in \{0,1\}^n\,:\, &\sum_{i = 1}^ka_ix_i \le b^k;\, x_i = 0,\,\forall\,i \in N_0^k;\\
&\,x_i = 1,\,\forall\, i \in N_1^k \Bigg\},
\end{aligned}
\end{equation}
where $N_0^k=N_0\backslash \{\ell+1,\ldots,k\}$, $N_1^k=N_1\backslash \{\ell+1,\ldots,k\}$, $\beta^k = \beta^\ell + \sum_{i\in (N_1\backslash N_1^k)} \alpha_i$ and $b^k=b - \sum_{i\in  N_1^k} a_i$.
Here, $N_0^k$ and $N_1^k$ denote the sets of variables fixed to $0$ and $1$, respectively, at the lifting step $k$.

To simplify the expression of lifting process, we introduce $F_k(z)$ to represent the following knapsack problem (KP):
\begin{equation}\label{fkz}
F_k(z) = \max_{\boldsymbol{x} \in \{0,1\}^k}\left\{\sum_{i=1}^k \alpha_i x_i:\sum_{i=1}^k a_i x_i \le z\right\},
\end{equation}
which involves the first $k$ items and capacity $z\le b$.
Now, we proceed to lift the $(k+1)$-th variable of the knapsack set in \eqref{kp}.
We categorize the lifting process for variable $x_{k+1}$ into two cases:
\begin{itemize}
	\item Up lifting. 
	If \( k+1 \in N_0^k \), the lifted inequality is
 $\sum_{i=1}^k \alpha_i x_i + \alpha_{k+1} x_{k+1} \le \beta^k$ 
with coefficient
\begin{equation}\label{uplift}
\begin{aligned}
\alpha_{k+1} &= \min_{\boldsymbol{x} \in \{0,1\}^k} \left\{\beta^k - \sum_{i=1}^k \alpha_i x_i : \, \sum_{i=1}^k a_i x_i \le b^k - a_{k+1}\right\}\\
&= \beta^k - F_k(b^k - a_{k+1}).
\end{aligned}
\end{equation}
Parameters update as
$N_0^{k+1} = N_0^k\backslash \{k+1\},~N_1^{k+1} = N_1^k,~b^{k+1} = b^k$ and $\beta^{k+1} = \beta^k.$
Note that if \(b^k < a_{k+1}\), the problem is infeasible and we set \(\alpha_{k+1} = 0\).
	\item Down lifting. 
	If \( k+1 \in N_1^k \), the lifted inequality is $\sum_{i=1}^k \alpha_i x_i + \alpha_{k+1} (x_{k+1} - 1) \le \beta^k,$
and coefficient
\begin{equation}\label{downlift}
\begin{aligned}
\alpha_{k+1} &= \max_{\boldsymbol{x} \in \{0,1\}^k} \left\{ \sum_{i=1}^k \alpha_i x_i - \beta^k :\, \sum_{i=1}^k a_i x_i \le b^k + a_{k+1}\right\}\\
&= F_k(b^k + a_{k+1}) - \beta^k.
\end{aligned}
\end{equation}
Parameters update as
$N_0^{k+1} = N_0^k,~ N_1^{k+1} = N_1^k \setminus \{k+1\},~b^{k+1} = b^k + a_{k+1}$ and $\beta^{k+1} = \beta^k +\alpha_{k+1}.$
Unlike up lifting, this problem is always feasible since \( b^k + a_{k+1} > 0 \).
\end{itemize}
Then we can get valid inequality $\sum_{i=1}^{k+1} \alpha_i x_i\le \beta^{k+1}$ for $\mathcal{X}(N_0^{k+1},N_1^{k+1})$.
The procedure proceeds iteratively, following the same lifting steps, until a valid inequality for $\mathcal{X}$ is derived.

The sequential lifting process generally requires solving $\left(|N_0| + |N_1| \right)$ KPs.
Since each step only differs from the previous one by a single item, \citet{vasilyev2016implementation} proposed a method to avoid redundant computations across iterations, rather than solving each problem independently.

The DP method computes optimal values $F_k(z)$ for the first $k$ items and each capacity $z \in [0,b]$. 
This allows direct evaluation of whether to include item \( (k+1) \), leading to the calculation of \( F_{k+1}(z) \).
Initializing \(F_0(z) = 0\) for all \(z \in [0,\,b ]\), the recursive formulation is given by
\begin{equation}\label{bellman}
\begin{aligned}
F_{k+1}(z) = 
\left\{\begin{aligned}
&F_{k}(z), &\text{if } z < a_{k+1};\\
&\max\{ F_{k}(z), F_{k}(z-a_{k+1}) + \alpha_{ k+1}\},&\text{if } z \ge a_{k+1}.
\end{aligned}
\right.
\end{aligned}
\end{equation}
This is the Bellman recursion \citep{bellman1957comment}, which is fundamental to the DP method for computing $F_n(b)$.
At iteration $k$, given $F_k(z)$ for all $z\in [0,\,b]$, the lifting coefficient $\alpha_{k+1}$ can be obtained from \eqref{uplift} or \eqref{downlift}.
After determining $a_{k+1}$ and $\alpha_{k+1}$, the recursion in \eqref{bellman} is used to update $F_{k+1}(z)$ for all $z\in [0,b]$. 
If all $a_i \in \mathbb{Z}$ and $b\in \mathbb{Z}$, the overall time complexity of the DP algorithm is \(\mathcal{O}(nb)\), and the space complexity is $O(b)$.
Since it is not necessary to store the full solution vector or intermediate results at each iteration, the memory requirement does not exceed $\mathcal{O}(b)$.

The following example provides the lifting procedure using DP with arrays.
\begin{example}\label{ex}
		Consider the binary knapsack set
	\begin{equation}\nonumber
	\mathcal{X} =\left\{\boldsymbol{x} \in \{0,1\}^7\,:\, 3x_1+4x_2+5x_3+4x_4+2x_5	+3x_6+6x_7\le 18\right\}.
	\end{equation}
	Let $C=\{1,2,3\},~N_0 = \{4,6\},~N_1=\{5,7\}$ and the lifting sequence $\{4,5,6,7\}$. 
	The seed inequality $x_1+x_2+x_3 \le 2$ is valid for the set
	\begin{equation}\nonumber
	\begin{aligned}
	\mathcal{X}(N_0,N_1) =\Big\{ \boldsymbol{x} \in \{0,1\}^n\,:\,
	&3x_1+4x_2+5x_3 	\le 18 - (2+6)=10,\\
& x_4=0,\, x_5=1,\,x_6=0,\,x_7 =1\Big\}.
	\end{aligned}
	\end{equation}
	We now have 
	$N_0^3 =\{4,6\}, ~N_1^3=\{5,7\},~b^3= 10,~\beta^3 = 2$.
	 Assume the values of $F_3(z)$ for all \( z =[ 0, \, 18] \) are precomputed as shown in Table \ref{table1}.
	When variable $x_4\in N_0^3$ is up-lifted, applying \eqref{uplift} gives the coefficient $\alpha_4=\beta^3 - F_3(b^3 - a_{4}) = 2-F_3(6)= 1$. 
	 Then the value of \( F_4(z) \) is obtained by applying the DP recursion \eqref{bellman}, as shown in Table \ref{table2}.
		\begin{table*}[h]
				\caption{DP array of $(z,F_3(z))$}
			\centering
		\begin{tabular}{|c|c|c|c|c|c|c|c|c|c|c|}
			\hline
			$z$ & \textcolor{blue}{0} & 1 & 2 & \textcolor{blue}{3} &4& 5 & 6 &  \textcolor{blue}{7}& {8} &9\\ \hline
			$F_3(z)$ & \textcolor{blue}{0} & 0 & 0 &  \textcolor{blue}{1} & 1&1 &  1 &\textcolor{blue}{2}&2 & 2 \\ \hline
			$z$ &10 & 11  &  \textcolor{blue}{12} & 13 & 14 & {15}& 16 & 17  &18&\\ \hline
			$F_3(z)$ & 2& 2&  \textcolor{blue}{3} & 3 & 3&3 &3 & 3&3 & \\ \hline
		\end{tabular}
		\label{table1}
	\end{table*}

		\begin{table*}[h]
			\centering
				\caption{DP array of $(z,F_4(z))$}
			\begin{tabular}{|c|c|c|c|c|c|c|c|c|c|c|}
			\hline
			$z$ & \textcolor{blue}{0} & 1 & 2 & \textcolor{blue}{3} &4& 5 & 6 &  \textcolor{blue}{7}& 8 &9\\ \hline
			$F_4(z)$ & \textcolor{blue}{0} & 0 & 0 &  \textcolor{blue}{1} & 1&1 &  1& \textcolor{blue}{2} &2& 2 \\ \hline
			$z$ &10 &  \textcolor{blue}{11} & {12} & 13 & 14 & 15 & \textcolor{blue}{16} & 17  &18&\\ \hline
			$F_4(z)$ & 2&  \textcolor{blue}{3}&3& 3 & 3 & 3  &\textcolor{blue}{4} & 4&4 & \\ \hline
		\end{tabular}
		\label{table2}
	\end{table*}
	Updating the parameters gives $N_0^4 = \{6\},~ N_1^4=\{5,7\},~ b^4 = 10$ and $\beta^4 = 2$.
	Therefore, the valid inequality for the set $\mathcal{X}(N_0^4,N_1^4)$ becomes $x_1+x_2+x_3 +x_4\le 2$.
\end{example}

The DP method typically requires retaining the entire array, which is often impractical in the lifting process, particularly for large capacities.
Example~\ref{ex} illustrates two drawbacks of using DP with arrays:
(i)  as shown by the blue columns in Tables \ref{table1} and \ref{table2}, 
all values $F_3(z)$ and $F_4(z)$ for $z \le 18$ can be recovered using only the values in the blue entries, making the maintenance of the remaining (black) columns redundant;
(ii)  the DP array requires that both weights and capacity to be integers, which is a limitation of using the array data structure.

Therefore, our main idea is to replace the full DP array with a compact list that records only the breakpoints where the array values increase. 
Specifically, we focus on the columns where the values increase (i.e. the blue columns in Tables \ref{table1} and \ref{table2}) in the following discussion.

%% file: tex/method.tex
\section{Sequential lifting with dominance list}\label{sec3}
In this section, we exploit dominance list structure into the sequential lifting process. 
In Section \ref{definition}, we give the definitions of state and dominance list. 
Then we introduce the algorithm for merging the dominance list in Section \ref{generation}.
Section \ref{reduction} introduces the reduction method that may further shorten the length of the dominance list and DP array.
Finally, Section \ref{lifting} integrates these components into the complete framework of an exact sequential lifting algorithm.

\subsection{State and dominance list}\label{definition}
For the sake of completeness, we recall the following definition as given by \citet{kellerer2004multidimensional}.
\begin{definition}\label{def:state}
	 For the KP $F_k(z)$ in \eqref{fkz}, each pair $(z, F_k(z))$ in the DP array is called a \textit{state} $(w, p)$, 
	where $w = \sum_{i=1}^k a_i x_i$ and $p = \sum_{i=1}^k \alpha_i x_i$ for some $x \in \{0,1\}^k$.
\end{definition}

For each $k= 1, \ldots, n$, the set of states can be expressed as a list
\begin{equation}\label{List}
\bar{L}_k = \left\{(w,p):\, w = \sum_{i=1}^k a_ix_i,~p = \sum_{i=1}^k \alpha_ix_i,~\forall~x\in\{0,1\}^k\right\}.
\end{equation}
The list can be pruned by exploiting dominance.

\begin{definition}\label{dom}
	For  $(w',p'),\,(w'',p'')\in \bar{L}_k$, we say that state $(w',p')$ dominates state $(w'',p'')$ if $w' <w''$ and $p'\ge p''$ or if $w' \le  w''$ and $p' > p''$.
\end{definition}
A dominated state cannot contribute to any optimal solution, which motivates the following definition.

\begin{definition}\label{def:DT}
	A \textit{dominance list} $L_k$ is defined as
	\begin{equation}\nonumber
	\begin{aligned}
	L_k = \bar{L}_k\backslash\Big\{(w,p)\in \bar{L}_k:\, &(w,p) \text{ is dominated by } (w',p') \text{ for some } \\
	&(w',p')\in \bar{L}_k\Big\}, 
	\end{aligned}
	\end{equation}
	where $\bar{L}_k$ is a list in \eqref{List}.
\end{definition}

In the dominance list $L_k$, all states are mutually non-dominated, which implies that the states can be ordered, as shown below.

\begin{lemma}
	For any two distinct states $(w',p')$ and $(w'',p'')$ in a dominance list $L_k$, either $w'<w''$ and $p'<p''$ or $w'>w''$ and $p'>p''$.\label{order}
\end{lemma}
\begin{proof}
	 We only prove that there are no state dominated by other different state.
	 \begin{itemize}
	 	\item If $w' = w''$ and $ p' < p'' $, then $(w', p')$ is dominated by $(w'',p'')$.
	 	\item If $w' = w''$ and $ p' > p'' $, then $(w'', p'')$ is dominated by $(w',p')$.
	 	\item If $w' < w''$ and $ p' \geq p'' $, then $(w'', p'')$ is dominated by $(w',p')$.
	 	\item If $w' > w''$ and $ p' \leq p'' $, then $ (w', p') $ is dominated by $ (w'', p'') $.
	 \end{itemize}
	 The remaining cases are  $w'<w''$ and $p'<p''$ or $w'>w''$ and $p'>p''$.\qed
\end{proof}

Lemma~\ref{order} shows that any two distinct states in the dominance list $L_k$ satisfy a strict order.
Therefore, the dominance list $L_k$ can be written as
\begin{equation}\nonumber
	L_k = \left\{(w^j,p^j)\right\}^m_{j=1},\label{Lk} 
\end{equation}
which satisfies that 
\begin{equation}
	w^j < w^{j+1}~\text{and}~p^j < p^{j+1},~ \forall\,j = 1,\ldots,m-1.\label{ascending} 
\end{equation}
Hence, each dominance list is strictly increasing in both $w$ and $p$.
\begin{theorem}
	Let $L_k$ be the dominance list of $F_k(z)$ for all $z\in [0,b]$. 
	Then, for any $\bar{z}\in [0,b]$, 
	\begin{equation}
		F_k(\bar{z}) = \left\{\begin{array}{ll}
			p^j,& \text{if }~\bar{z} \in [w^j, w^{j+1}),~j=1,\ldots,m-1;\\
			p^m,& \text{if }~\bar{z} \in [w^m, b].\\
		\end{array}\right.
	\end{equation}
\end{theorem}
\begin{proof}
	For any \( \bar{z} \in [w^j, w^{j+1}) , ~j\in\{1,\ldots,m-1\}\), let $(w^j, p^j)\in L_k$ be the corresponding state by the Definition \ref{def:state}.
	Since \( w^j \leq \bar{z} \), this state is feasible for $F_k(\bar{z})$, so $F_k(\bar{z}) \geq p^j$.
	If $F_k(\bar{z}) > p^j$, then there exists a feasible state $(w',p')$ with $w' \leq \bar{z}$ and $p' > p^j$.  
	\begin{itemize}
		\item [1)] If \( (w',p')\in L_k\), 	then $w' < w^j$ and $p' > p^j$, contradicting the definition of dominance list.  
		\item [2)] If \( (w',p')\notin L_k\), it is dominated by some $(w'',p'') \in L_k$ with $w'' \leq w'$ and $p'' \geq p' > p^j$, reducing to the case 1).  
	\end{itemize}
	Hence, $F_k(\bar{z})=p^j$ for all $\bar{z}\in[w^j,w^{j+1})$.  
	
	The argument for $\bar{z}\in[w^m,b]$ is similar. 
	Since \( \bar{z} \geq w^m \), the state \( (w^m, p^m) \) is feasible, and any feasible solution with value larger than $p^m$ would again contradict the dominance property. 
	Therefore, $F_k(\bar{z})=p^m$.
	\qed
\end{proof}

Consequently, the dominance list provides a compact representation of the DP array by eliminating dominated entries.

\begin{example}
	Consider the set $\mathcal{X}$ from Example \ref{ex}.
	Applying the dominance properties described above, Table~\ref{table1} can be reduced to the dominance list of $L_3$ in Table~\ref{table4}.
		\begin{table}[!htb]
		\centering
		\caption{Dominance list of $L_3$}\label{table4}
		\renewcommand{\arraystretch}{1.5}
			\begin{tabular}{|c|c|c|c|c|}
			\hline
			$w$ & \textcolor{blue}{0} & \textcolor{blue}{3}& \textcolor{blue}{7}&  \textcolor{blue}{12}\\ \hline
			$p$ & \textcolor{blue}{0}  &  \textcolor{blue}{1} & \textcolor{blue}{2} & \textcolor{blue}{3}\\ \hline
		\end{tabular}
	\end{table}
\end{example}

\subsection{Dominance list iteration} \label{generation}
We now introduce the generation of $L_k$ for $k=1,\ldots,n$.
Initialize the dominance list as $L_{0}=\{(0,0)\}$.
Suppose the dominance list $L_{k-1}=\left\{(w^1,p^1),\ldots,(w^m,p^m)\right\}$. 
Our goal is to construct the dominance list $L_{k}$ for item $k$ with weight $a_{k}$ and coefficient $\alpha_{k}$.
Without loss of generality, assume $a_{k}>0,~\alpha_{k}>0$.
It is straightforward to verify that if $\alpha_{k} = 0$, then $L_{k}=L_{k-1}$;
if $a_{k} = 0$, then $L_{k}= \left\{ (w^1, p^1+\alpha_{k}),\ldots,(w^{m},p^{m}+\alpha_{k})\right\}$.

In the general case, $L_k$ is obtained from $L_{k-1}$ in two steps.
First, each state in $L_{k-1}$ is added by $(a_k,\alpha_k)$ to form a new list $L_{k-1}'$, i.e.,
\begin{equation}\nonumber
\begin{aligned}
L_{k-1}':= & L_{k-1} \oplus (a_k, \alpha_k) \\
= &\left\{(w^1 + a_k, p^1+ \alpha_k), \ldots,(w^m+a_k,p^m+ \alpha_k)\right\}.
\end{aligned}
\end{equation}
Second, the two lists $L_{k-1}$ and $L_{k-1}'$ are merged to form the new list $L_k$. 
Formally, merging two dominance lists \( L_{k-1} \) and \( L'_{k-1} \) into a new list \( L_k \) under the capacity limit, proceeds as follows:
\begin{enumerate}
	\item \textbf{Initialization}: Begin with the initial state \( (w', p') := (0, -\infty) \) in \( L_k \), which serves to track the most recent non-dominated state.
	\item \textbf{Iterative merging}: 
	Traverse and evaluate both lists $L_{k-1}$ and $L'_{k-1}$ in order of increasing weight. 
	Discard any state $(w,p)$ with $w>b$.
	\item \textbf{Dominance check}: 
	For each candidate state $(w,p)$ with $w \geq w'$:
	\begin{itemize}
		\item \textbf{Case 1}: If \( p \leq p' \), then \( (w, p) \) is dominated or identical to \( (w', p') \).
		In this case, \( (w, p) \) is discarded.  
		\item \textbf{Case 2}: If \( p > p' \) and \( w = w' \), then \( (w', p') \) is dominated by \( (w, p) \). 
		Replace $(w',p')$ with $(w,p)$ in $L_k$.  
		\item \textbf{Case 3}: If \( p > p' \) and \( w > w' \), then $(w,p)$ is non-dominated and added to $L_k$.  
	\end{itemize}
	\item \textbf{Output}: The resulting list \( L_k \) contains only feasible and non-dominated states within capacity.
\end{enumerate}
This merging procedure ensures that $L_k$ is compact and contains only non-dominated states, guaranteeing efficiency and correctness in subsequent iterations.
It is important to note that each dominance list is totally ordered \eqref{ascending}. 
This structural property promotes efficient merging by reducing unnecessary comparisons, as detailed in Alg. \ref{alg:merge}.

\begin{algorithm}[H]
	\caption{Merge dominance lists}
	\label{alg:merge}
	\begin{algorithmic}[1]
		\Require Dominance list $L_{k-1}=\left\{(w^1,p^1),\ldots,(w^m, p^m)\right\}$, weight $a_k$, lifting coefficient $\alpha_k$ and capacity $b$.
		\Ensure Updated dominance list $L_k=\left\{(w^1,p^1),\ldots,(w^s, p^s)\right\}$.
		\State Initialize $i \leftarrow 1$, $j \leftarrow 1$, $s \gets 1$, $(w^s,p^s):= (w^1, p^1)$
		\While{$i \le m$ \textbf{and} $j \le m$ \textbf{and} $w^s \le b$ }
		\If{$w^i \leq (w^j + a_k)$}
		\If{$p^i >p^s$}
		\If{$w^i > w^s$}
		\State $s\leftarrow s+1$
		\EndIf
		\State $w^s \leftarrow w^i,~ p^s \leftarrow p^i$
		\EndIf
		\State $i \leftarrow i+1$
	\Else
		\If{$(p^j + \alpha_k) > p^s$}
	\If{$(w^j + a_k)> w^s$}
	\State $s\leftarrow s+1$
	\EndIf
	\State $w^s \leftarrow(w^j + a_k),~ p^s \leftarrow (p^j + \alpha_k)$
	\EndIf
	\State $j \leftarrow j + 1$
	\EndIf
		\EndWhile
			\While{$j \le m$ \textbf{and} $w^s \le b$ }
			\If{$(p^j + \alpha_k) > p^s$ }
			\If{$(w^j + a_k) > w^s$}
			\State $s\leftarrow s+1$
			\EndIf
			\State $w^s \leftarrow (w^j + a_k),~ p^s \leftarrow (p^j + \alpha_k)$
			\EndIf
			\State $j \leftarrow j + 1$
			\EndWhile
		\If{$w^s> b$}
		\State $s \leftarrow s-1$
		\EndIf
	\end{algorithmic}
\end{algorithm}

\begin{example}
	Consider the set $\mathcal{X}$ from Example \ref{ex}.
	Using Alg. \ref{alg:merge}, we merge the dominance list $L_3$ with list  $L_3\oplus (a_4,\alpha_4)$, as shown in Tables \ref{table5} and \ref{table6}, respectively.
	Table \ref{table7} shows the resulting updated dominance list $L_4$, which shows a shorter length compared to the DP array of \( (z, F_4(z)) \) in Table \ref{table2}.
	\begin{table*}[h]
		\centering
		\begin{minipage}[c]{0.5\textwidth}
			\caption{Dominance list of $L_3$}\label{table5}
			\centering
			\begin{tabular}{|c|c|c|c|c|}
				\hline
				$z$ & {0} &{3}& {7}&  {12}\\ \hline
				$F_3(z)$ & {0}  &  {1} &{2} & {3}\\ \hline
			\end{tabular}
		\end{minipage}\hfill
		\begin{minipage}[c]{0.5\textwidth}
			\caption{List of $L_3\oplus (a_4,\alpha_4)$}\label{table6}
			\centering
			\begin{tabular}{|c|c|c|c|c|}
				\hline
				$z+a_4$ & {4} & {7}&{11}& 16\\ \hline
				$F_3(z)+ \alpha_4$ & {1}  &{2} & {3} &{4}\\ \hline
			\end{tabular}
		\end{minipage}		
	\end{table*}
	
	\begin{table*}[h]
		\caption{Dominance list of $L_4$}\label{table7}
		\centering
		\renewcommand{\arraystretch}{1.5}
		\begin{tabular}{|c|c|c|c|c|c|c|c|}
			\hline
			$z$ & 0& 3&(4)&7&11&(12)&16\\ \hline
			$F_4(z)$ & 0&1&(1)&2&3&(3)&4\\ \hline
		\end{tabular}
	\end{table*}
\end{example}

The time complexity of the merging procedure is \(\mathcal{O}(b)\), since neither list exceeds length \(b\). 
Consequently, the overall complexity of Alg. \ref{alg:merge} is \( \mathcal{O}(nb)\), matching that of DP with arrays (if $b\in \mathbb{Z}$). 
The important distinction, however, lies in the best-case behavior: 
while both the best- and worst-case complexities of DP with arrays are typically \( \mathcal{O}(nb)\), 
Alg. \ref{alg:merge} can achieve a best-case time complexity of \( \mathcal{O}(n) \). 

Furthermore, due to the properties of DP with arrays, both the weights and the capacity must be integers.
In contrast, Alg. \ref{alg:merge} imposes no such restriction, enabling it to handle a broader class of KPs with binary knapsack sets.
Moreover, if the knapsack constraint ($b\in \mathbb{Z}$) is scaled by a constant factor \( C>1\), the length of the DP array is expanded by \( C \); 
in comparison, the length of the list in Alg. \ref{alg:merge} remains unaffected.
A detailed comparison and performance will be provided in the following numerical experiments section.

\subsection{Reduction method} \label{reduction}
During the lifting process, the values of \( F_k(z) \) are derived from the previously computed \( F_{k-1}(z) \).
However, not all values of $F_{k-1}(z)$ are needed; in fact, the relevant capacities $z$ lie within a restricted range. 
The method described in this section explicitly identifies such ranges, thereby reducing the length of dominance list and DP array.
Without loss of generality, we focus on the lifting step at iteration $k+1$. 
The following two observations illustrate the relationships between the requirements between consecutive steps.

\begin{observation}\label{ob1} 
	 To calculate the coefficients $\alpha_{k+1}$, 
	\begin{itemize}
		\item for up lifting \eqref{uplift}, the required value is $F_{k}( b^k - a_{k+1})$. 
		Noting that $b^k = b - \sum_{i\in  N_1^{k}} a_i$ and $N_1^{k} = N_1^{k+1}$, we have
		$F_{k}( b^k - a_{k+1}) = F_{k}( b - \sum_{i\in  N_1^{k+1}} a_i - a_{k+1})$;
		\item for down lifting \eqref{downlift}, the required value is $F_{k}( b^k + a_{k+1})$.  
		Since $b^k = b - \sum_{i\in  N_1^{k}} a_i$ and $N_1^{k} = N_1^{k+1}\cup \{k+1\}$, we obtain
		$F_{k}( b^k + a_{k+1}) = F_{k}( b - \sum_{i\in  N_1^{k+1}} a_i)$.
	\end{itemize} 
	\end{observation}

\begin{observation}\label{ob2}
	To compute $F_{k+1}(z)$, the DP recursion \eqref{bellman} requires both $F_{k}(z)$ and $F_{k}(z-a_{k+1})$ whenever $z\geq a_{k+1}$.
	\end{observation}

We now extend these observations to subsequent lifting steps.

\begin{theorem}
	Consider the sequential lifting process for knapsack set $\mathcal{X}$.
	 If $F_k(z)$ is known for all $z\in [(b - \sum_{i=k+1}^n a_i)^+,\,b]$, then the coefficients $\alpha_{k+1},\ldots,\alpha_n$ can be calculated by \eqref{uplift}, \eqref{downlift} and \eqref{bellman}, where $(c)^+ = \max\{c,0\}$.
\end{theorem}
\begin{proof}
	We first consider the $k+1$ lifting step.
	\begin{itemize}
		\item [1)] From Observation~\ref{ob1} and the inequality
		\begin{equation}\nonumber
		b - \sum_{i\in  N_1^{k+1}} a_i  \geq b - \sum_{i\in  N_1^{k+1}} a_i - a_{k+1} \geq b - \sum_{i=k+1}^n a_i,
		\end{equation}
		the coefficient $\alpha_{k+1}$ can be obtained from either \eqref{uplift} for up lifting or \eqref{downlift} for down lifting.
		\item [2)] Moreover, for all $z \in [(b - \sum_{i=k+2}^n a_i)^+,\,b]$, the values $F_{k+1}(z)$ can be computed by the recursion \eqref{bellman}, as stated in Observation~\ref{ob2}.
	\end{itemize}
	 Then, by recursively applying steps 1) and 2), the coefficients $\alpha_{k+2}, \ldots, \alpha_n$ can be determined exactly. 
	\qed
\end{proof}

\begin{corollary}\label{1}
	The problem $F_k(z)$ can only consider $z\in [(b - \sum_{i=k+1}^n a_i)^+,b] $ without affecting the computation in the process of calculating the lifting coefficients $\alpha_{k+1},\ldots,\alpha_n$.
\end{corollary}

After each step, when a new dominance list or DP array is generated, 
Corollary~\ref{1} can be applied to reduce its length.
The dominance list using reduction method is shown in the following Alg. \ref{alg:merge2}.

\begin{algorithm}[H]
	\caption{Merge dominance lists with the reduction method}
	\label{alg:merge2}
	\begin{algorithmic}[0]
		\Require Dominance list $L_{k-1}=\left\{(w^1,p^1),\ldots,(w^m, p^m)\right\}$, weight $a_k$, lifting coefficient $\alpha_k$, capacity $b$ and restricted value $r=(b - \sum_{i=k+1}^n a_i)^+$.
		\Ensure Updated dominance list $L_k=\left\{(w^1,p^1),\ldots,(w^s, p^s)\right\}$.
		\State // Perform the same steps as in Alg. \ref{alg:merge}, except for the following modifications:
		\State {\footnotesize{ 5:}} \textbf{if} {$w^i > w^s$	and $w^i \ge r$} \textbf{then}
		\State {\footnotesize{13:}} \textbf{if}  {$(w^j + a_k)> w^s$ and $(w^j + a_k) \ge r$}  \textbf{then}
		\State {\footnotesize{23:}} \textbf{if} {$(w^j + a_k) > w^s$ and $(w^j + a_k) \ge r$} \textbf{then}
	\end{algorithmic}
\end{algorithm}

When applying reduction method to the DP arrays, the lifting step proceeds as follows.
Specifically, at the \(k\)-th lifting step:
\begin{itemize}
	\item[1).] Apply Corollary \ref{1} to determine the restricted value $r=(b - \sum_{i=k+1}^n a_i)^+$ ;
	\item[2).] Update the DP array $(z,F_k(z))$ only for $z\ge r$ using the DP recursion \eqref{bellman}.
\end{itemize}
This procedure ensures that only essential values are computed when using the reduction method in DP with arrays.

\subsection{Exact sequential lifting algorithm} \label{lifting}
We present an exact sequential lifting algorithm for binary knapsack set,  based on the dominance list and reduction method.
Alg. \ref{alg:exact} outlines the procedure, which begins with a seed inequality, from which the coefficients of the remaining variables are derived through sequential lifting.  
For each variable $k$ in the fixed seed set $\{1, \ldots, \ell\}$, the dominance list $L_k$ is updated using the merge procedure and reduction method described in Alg. \ref{alg:merge2}. 
The algorithm then proceeds to lift the remaining variables $k \in \{\ell+1, \ldots, n\}$ in a fixed order, applying either up or down lifting based on the variable’s fixing status.
Once the lifting sequence is completed, the resulting lifted inequality is guaranteed to be valid for the entire knapsack set $\mathcal{X}$. 

\begin{algorithm}[t]
	\caption{Exact sequential lifting algorithm}
	\renewcommand{\algorithmicrequire}{\textbf{Input:}}
	\renewcommand{\algorithmicensure}{\textbf{Output:}}
	\label{alg:exact}
	\begin{algorithmic}[1]
		\Require  Inequality $\sum_{i=1}^{n}a_i x_i \le b$ for the knapsack set $\mathcal{X}$, seed inequality $\sum_{i=1}^{\ell}\alpha_i x_i \le \beta^\ell $ with the fixed index sets $N_0, N_1$, and lifting sequence $\{\ell+1, \dots, n\}$.
		\Ensure Valid inequality \(\sum_{i=1}^n \alpha_i x_i \leq \beta^n\) for the knapsack set $\mathcal{X}$.
		\State Initialize $F_0(z)=(0,0)$ for all $z \in[0,b]$ and set $N_0^\ell = N_0,~N_1^\ell = N_1$
		\For{\(k = 1, \dots, \ell\)}
		\State Update $L_k$ using Alg. \ref{alg:merge2}
		\EndFor
		\For{\(k = \ell+1, \dots, n\)}
		\If{\(k \in N_0^{k-1}\) (up lifting)}
		\State Compute $\alpha_{k} = \beta^{k-1} - F_{k-1}(b^{k-1}- a_{k})$
		\State Update \(N_0^{k} = N_0^{k-1} \backslash \{k\}\), \(N_1^{k} = N_1^{k-1}\), \(\beta^{k} = \beta^{k-1}\) and \(b^{k} = b^{k-1}\)
		\ElsIf{\(k \in N_1^{k-1}\) (down lifting)}
		\State Compute $\alpha_{k} = F_{k-1}(b^{k-1} + a_{k}) - \beta^{k-1}$
		\State Update \(N_0^{k} = N_0^{k-1}\), \(N_1^{k} = N_1^{k-1} \backslash \{k\}\), \(\beta^{k} = \beta^{k-1} + \alpha_{k}\) and \(b^{k} = b^{k-1} + a_{k}\)
		\EndIf
		\If{\(k < n\)}
		\State Update $L_k$ using Alg. \ref{alg:merge2}
		\EndIf
		\EndFor
	\end{algorithmic}
\end{algorithm}

Next, we build upon the results discussed above to provide a detailed lifting process based on Example~\ref{ex}, 
demonstrating how Alg. \ref{alg:exact} updates the dominance list. 
Additionally, we offer a direct comparison between Alg. \ref{alg:exact} and DP with arrays during sequential lifting procedure.
The lifting coefficients are derived step by step, ensuring that the resulting lifted inequality is valid for the entire knapsack set.

\begin{example}\label{ex2}
	Consider the binary knapsack set
	\begin{equation}\nonumber
	\mathcal{X} =\left\{\boldsymbol{x} \in \{0,1\}^7\,:\, 3x_1+4x_2+5x_3+4x_4+2x_5+3x_6+6x_7\le 18\right\}.
	\end{equation}
	Let \( C = \{1, 2, 3\}, N_0 = \{4, 6\}, N_1 = \{5, 7\} \) and the lifting sequence \( \{4, 5, 6, 7\} \). 
	The seed inequality $x_1+x_2+x_3 \le 2$ is valid for the set
	\begin{equation}\nonumber
	\begin{aligned}
	\mathcal{X}(N_0,N_1) =\Big\{ \boldsymbol{x} \in \{0,1\}^n\,:\,
	&3x_1+4x_2+5x_3 	\le 18 - (2+6)=10,\\
	& x_4=0,\, x_5=1,\,x_6=0,\,x_7 =1\Big\}.
	\end{aligned}
	\end{equation}
	We now have 
	$N_0^3 =\{4,6\}, ~N_1^3=\{5,7\},~b^3= 10,~\beta^3 = 2$.
	
	First, we apply Alg. \ref{alg:exact} for the sequential lifting procedure.
	By Corollary \ref{1}, the dominance list $L_3$ can be restricted to $z \in
	\left[ (b-\sum_{i  =4}^7 a_i),b\right] = [3,18]$.
	Then the reduced dominance list $L_3$ is shown in Table \ref{table8}.
	\begin{table*}[!htb]
		\centering
		\begin{minipage}[c]{0.3\textwidth}
			\caption{Reduced dominance list of $L_3$}\label{table8}
			\centering
			\begin{tabular}{|c|c|c|c|}
				\hline
				$z$ &{3}& {7}&  {12}\\ \hline
				$F_3(z)$  &  {1} &{2} & {3}\\ \hline
			\end{tabular}
		\end{minipage}\hfill
		\begin{minipage}[c]{0.3\textwidth}
			\centering
			\caption{List of $L_3\oplus (a_4,\alpha_4)$}\label{table9}
			\begin{tabular}{|c|c|c|c|}
				\hline
				$z+a_4$ & {7}&{11}& {16}\\ \hline
				$F_3(z)+\alpha_4$   &{2} & {3} &{4}\\ \hline
			\end{tabular}
		\end{minipage}\hfill
		\begin{minipage}[c]{0.3\textwidth}
			\caption{Dominance list of $L_4$}\label{table10}
			\centering
			\begin{tabular}{|c|c|c|c|c|}
				\hline
				$z$ & 3&7&11&16\\ \hline
				$F_4(z)$ &1&2&3&4\\ \hline
			\end{tabular}
		\end{minipage}
	\end{table*}
	For variable $x_4$, the lifting coefficient is $\alpha_4 = \beta^3 - F_3(b^3-a_4) = 2-F_3(10-4) =1$,
	where \(F_3(6)=1\) since \(6\in[3,7)\) in Table~\ref{table8}. 
	After merging $L_3$ with $L_3\oplus (a_4,\alpha_4)$ (Table  \ref{table9}), the updated dominance list $L_4$ is obtained in Table \ref{table10}.
	The parameters are updated to $b^4 = 10,~\beta^4 =2$, $N_0^4=\{6\},$ and $ N_1^4=\{5,7\}$.
	Applying Corollary \ref{1}, $L_4$ can be restricted to $z =[(b-\sum_{i  =5}^7 a_i =)\,7, \,18]$, as shown in Table \ref{table11}.
	\begin{table*}[!htb]
		\centering
		\begin{minipage}[c]{0.5\textwidth}
			\caption{Reduced dominance list of $L_4$}\label{table11}
			\centering
			\begin{tabular}{|c|c|c|c|}
				\hline
				$z$ &7&11&16\\ \hline
				$F_4(z)$ &2&3&4\\ \hline
			\end{tabular}
		\end{minipage}\hfill
		\begin{minipage}[c]{0.5\textwidth}
			\caption{Dominance list of $L_5$}\label{table12}
			\centering
			\begin{tabular}{|c|c|c|c|c|}
				\hline
				$z$ & 7&9&13&18\\ \hline
				$F_5(z)$ &2&3&4&5\\ \hline
			\end{tabular}
		\end{minipage}
	\end{table*}
	\begin{table*}[!htb]
	\centering
	\begin{minipage}[c]{0.5\textwidth}
		\caption{Reduced dominance list of $L_5$}\label{table17}
		\centering
		\begin{tabular}{|c|c|c|c|}
			\hline
			$z$ & 9&13&18\\ \hline
			$F_5(z)$ &3&4&5\\ \hline
		\end{tabular}
	\end{minipage}\hfill
	\begin{minipage}[c]{0.5\textwidth}
		\caption{Reduced dominance list of $L_6$}\label{table14}
		\centering
		\begin{tabular}{|c|c|c|}
			\hline
			$z$ &13&18\\ \hline
			$F_6(z)$ &4&5\\ \hline
		\end{tabular}
	\end{minipage}
\end{table*}
	Next, for variable $x_5$, the coefficient is $\alpha_5= F_4(b^4+a_5) - \beta^4=F_4(10+2)-2= 1$.
	Merging \( L_4 \) with \( L_4 \oplus (a_5, \alpha_5) \) gives the dominance list \( L_5 \) (Table \ref{table12}), and updates parameters $b^5 = 12,~\beta^5 =3,~N_0^5=\{6\},~N_1^5=\{7\}$.
	And the reduced dominance list of $L_5$ is shown in Table \ref{table17}.
	For variable $x_6$, the coefficient is $\alpha_6=\beta^5 -F_5(b^5-a_6)=3-F_5(12-3)=0$, indicating no contribution from \( x_6 \).  
	Thus, \( L_6 = L_5 \). 
	Update $b^6 = 12,~\beta^6 =3$ and $N_0^6=\emptyset,~N_1^6=\{7\}$.
	After applying Corollary \ref{1} with \( z \in [12,18] \), the reduced dominance list of $L_6$ is given in Table \ref{table14}.  
	Finally, for variable $x_7$, $\alpha_7 = F_6(b^6+a_7) - \beta^6 = F_6(18) - 3 = 2$.
	Therefore, the final sequentially lifted inequality for \( \mathcal{X} \) is
	$x_1+x_2+x_3 + x_4+x_5 + 2x_7 \le 5.$
	
	Next, we apply the reduction method to the DP arrays during each lifting procedure (Tables~\ref{table20}--\ref{table23}).
	The reduced DP array retains only the values necessary for the subsequent lifting steps, significantly shortening its length; for instance, in Table~\ref{table23}, the length decreases from $18$ to $6$.
	Nevertheless, even with reduction, the DP arrays remain longer than the corresponding reduced dominance list.
		\begin{table*}[!htb]
				\caption{Reduced DP array of $(z,F_3(z))$}\label{table20}
			\centering
			\begin{tabular}{|c|c|c|c|c|c|c|c|c|c|c|c|c|c|c|c|c|}
				\hline
				$z$ &  3 &4& 5 & 6 &  7& {8} &9&10 & 11  & 12 & 13 & 14 & {15}& 16 & 17  &18\\ \hline
				$F_3(z)$ & 1 & 1&1 &  1 &2&2 & 2& 2& 2& 3 & 3 & 3&3 &3 & 3&3  \\ \hline
			\end{tabular}
		\end{table*}
		\begin{table*}[!htb]
		\caption{Reduced DP array of $(z,F_4(z))$}\label{table21}
		\centering
		\begin{tabular}{|c|c|c|c|c|c|c|c|c|c|c|c|c|}
			\hline
			$z$ &   7& {8} &9&10& 11  & 12 & 13  & 14 & {15}& 16 & 17  &18\\ \hline
		$F_4(z)$ &2&2 & 2& 2 &3&3& 3& 3 & 3 & 4&4&4\\ \hline
		\end{tabular}
	\end{table*}
		\begin{table*}[!htb]
		\centering
		\begin{minipage}[c]{0.5\textwidth}
			\caption{Reduced DP array of $(z,F_5(z))$}\label{table22}
		\centering
		\begin{tabular}{|c|c|c|c|c|c|}
			\hline
			$z$    &9&10& 11  & 12 & 13 \\ \hline
			$F_5(z)$ &3&3& 3& 3 & 4\\ \hline
			$z$  & 14 & {15}& 16 & 17  &18\\ \hline
			$F_5(z)$  & 4&4&4&4&5\\ \hline
		\end{tabular}
		\end{minipage}\hfill
		\begin{minipage}[c]{0.5\textwidth}
			\caption{Reduced DP array of $(z,F_6(z))$}\label{table23}
		\centering
		\begin{tabular}{|c|c|c|c|c|c|c|}
			\hline
			$z$    & 13 & 14 & {15}& 16 & 17  &18\\ \hline
			$F_6(z)$ & 4& 4&4 &4&4&5\\ \hline
		\end{tabular}
		\end{minipage}
	\end{table*}
	
\end{example}

This example provides a detailed, step-by-step illustration of the sequential lifting procedure, 
showing how Alg.~\ref{alg:exact} updates the dominance list and applies the reduction method to both dominance list and DP array. 
Building on these observations, the next section presents a comprehensive computational study to further evaluate the proposed algorithm.

%% file: tex/result.tex
\section{Numerial Results}\label{numerial}
In this section, we present a series of computational experiments to evaluate the performance of the proposed exact sequential lifting algorithm for binary knapsack set.
To do this, we first perform experiments to demonstrate the effectiveness of the proposed algorithm across binary integer programming problems with knapsack sets of different sizes and capacities, as presented in Section \ref{sec4.1}.
Then, we integrate the proposed algorithm into the exact separation framework and present computational results to examine its scalability and numerical stability when the scale of binary knapsack set increases, as discussed in Section \ref{sec4.2}.
Finally, we present computational results to compare the performance with approximate algorithm in the MIP solver to further evaluate its computational efficiency, as outlined in Section \ref{sec4.3}. 
The proposed algorithm has been implemented in C/C++ and all numerical tests were implemented on a cluster of Intel(R) Xeon(R) Gold 6140 CPU @ 2.30 GHz computers, with 180 GB RAM, running Linux (in 64 bit mode).

To ensure a fair comparison, the DP with arrays method is obtained by replacing Steps 3 and 14 in Alg. \ref{alg:exact} with the DP recursion in \eqref{bellman}, which used for computing lifting coefficients and updates arrays.
Additionally, to assess the effect of the reduction method introduced in Section \ref{reduction}, we use four different settings based on the lifting method and whether reduction is applied:
\begin{itemize}
	\item \textbf{DL}: Alg. \ref{alg:exact} without the reduction method (i.e., replace Alg. \ref{alg:merge2} with Alg. \ref{alg:merge} in Steps 3 and 14);
	\item \textbf{DP}: DP with arrays without the reduction method (i.e., replace Steps 3 and 14 in Alg. \ref{alg:exact});
	\item \textbf{DL-R}: Alg. \ref{alg:exact} (with the reduction method);
	\item \textbf{DP-R}: DP with arrays with the reduction method applied (i.e., replace Steps 3 and 14 in Alg. \ref{alg:exact} and use reduction method in Section \ref{reduction}).
\end{itemize}

\subsection{Comparison of the DL, DP and reduction method}\label{sec4.1}
The goal of experiment is to compare and evaluate the performance of the four settings on binary knapsack sets.
First, we perform numerical experiments to show the impact of the number of items and capacity.
Then we present the computational results to analyze the impact of weight and capacity.
The proposed algorithm has been linked with IBM ILOG CPLEX optimizer 20.1.0.0\footnote{https://www.ibm.com/cn-zh/products/ilog-cplex-optimization-studio/cplex-optimizer} library for solving linear programming (LP) problems.
All other ILOG CPLEX parameters are set to default values unless otherwise stated.

\subsubsection{Testset}
Our computational tests are performed on the random binary integer programming problem of the following form:
\begin{equation}\label{mip}
\max \{\boldsymbol{1}_n^\top \boldsymbol{x} : A\boldsymbol{x} \le  \boldsymbol{1}_m b,~\boldsymbol{x}\in \{0,1\}^n\},
\end{equation}
where $A = [a_{ij}]\in \mathbb{Z}_+^{m\times n}, ~b\in \mathbb{Z}_+$, \(\boldsymbol{1}_n\) and \(\boldsymbol{1}_m\) denote all-one column vectors of dimension $n$ and $m$, respectively.
The number of rows $m$ is fixed at 1000.

To capture the relationship between the capacity \(b\) and the total weight of each row, we introduce a parameter $\lambda$, the ratio of \(b\) and the total weight of each row.
Specially, each weight is defined as $a_{ij} := \min \{\bar{a}_{ij},b\}$, where \(\bar{a}_{ij}\) is uniformly distributed to an integer in the range \([0, a_{\max}]\), and \(\lambda =\frac{2b}{n a_{\max}}\).
For generality, we assume $ a_{\max}\ge 1$, which implies that
\begin{equation}\label{lam}
\lambda \le \frac{2b}{n}.
\end{equation}

The lifting process is applied independently to each constraint of the problem.
For each constraint, we adopt the seed inequality and lifting order proposed by \citet{kaparis2008local}, summarized as follows:
\begin{itemize}
	\item [1).]  Solve the LP relaxation of the problem \eqref{mip} to obtain the fractional solution \( x^* \).
	Sort variables in non-increasing order and define index sets \( N^0 = \{j \in N : x_j^* = 0\} \) and \( N^1 = \{j \in N : x_j^* = 1\} \).
	\item [2).]  Identify a minimal cover set $C \subseteq S_i := \{ j \in N : a_{ij} > 0 \}$ to construct the seed inequality.
	\item [3).] Perform  sequential lifting in the following order:
	\begin{itemize}
		\item Up-lift variables in $S_i \setminus (C \cup N^0)$;
		\item Down-lift variables in $D := C \cap N^1$, i.e., variables that belong to the cover set and are fixed at one;
		\item Up-lift the remaining variables in \( (S_i \cap N^0) \setminus C \).
	\end{itemize}
\end{itemize}

To systematically evaluate the impact of the lifting process, we compare the total lifting time across all \(m\) knapsack constraints under four different settings in the following experiments,
rather than directly measuring the solution time of the problem \eqref{mip}.
All experiments are based on the solutions obtained from the initial LP relaxations.
Note that in the following tables reporting computational results, if condition \eqref{lam} is not satisfied, the corresponding entry is expressed as ``--".

\subsubsection{Performance and impact of the number of item and capacity}
In this subsection, we demonstrate the effectiveness of the four settings by reporting lifting times (in seconds) for 1000 knapsack constraints under different combinations of the number of items ($n$ = 100, 200, 500, 1000, 2000, 3000, 4000, 5000) and knapsack capacities (b = $10^3, 10^4, 10^5, 10^6$), with fixed \(\lambda = 0.6\).
For each combination, five independent instances are generated to reduce randomness and ensure robustness in performance comparisons.

Table \ref{random13} reports clear performance differences between DL and DP.
DL consistently outperforms DP across most tested values of $n$ and $b$. 
For instance, when \(n = 1000\), the running time of DL increases only slightly from 5.21\,s at \(b = 10^4\) to 5.45\,s at \(b = 10^6\),
while DP grows sharply from 77.00\,s to 7827.01\,s, under the same conditions. 
This difference demonstrates that DL scales primarily with $n$, whereas DP is sensitive to both $n$ and $b$, making DL much more suitable for large-capacity instances.  

Reduction methods further enhance improvements, as shown in Table \ref{random15}.
DL-R consistently improves upon DL, with the improvement becoming more pronounced as $n$ and $b$ increase. 
For example, when $n=2000$ and $b=10^5$, DL requires 21.58\,s, while DL-R reduces this to 13.56\,s.
Similarly, DP-R outperforms DP, highlighting the effectiveness of reducing redundant entries prior to DP with arrays. 
This demonstrates that the reduction method is particularly effective for both methods.

\begin{table*}[thb]
	\caption{Running time for settings without reduction method under $\lambda = 0.6$ }
	\footnotesize
	\centering
	\label{random13}
	\begin{tabular}{|c|cc|cc|cc|cc|}
		\toprule
		 & \multicolumn{2}{c|}{{$b=10^3$}} & \multicolumn{2}{c|}{{$b=10^4$}} & \multicolumn{2}{c|}{{$b =10^5$}} & \multicolumn{2}{c|}{{$b=10^6$}} \\
		\cmidrule{2-9}
		n & DL & DP & DL & DP & DL & DP & DL & DP\\
		\midrule
	100 &         0.06 &      0.75 &      0.05 &      7.35 &      0.08 &     73.81 &      0.05 &    751.18 \\
	200 &          0.22 &      1.55 &      0.21 &     15.06 &      0.23 &    149.59 &      0.23 &   1524.46 \\
	500 &        1.27 &      3.70 &      1.38 &     38.48 &      1.43 &    381.23 &      1.40 &   3867.50 \\
	1000 &          4.28 &      6.96 &      5.21 &     77.00 &      5.44 &    764.96 &      5.45 &   7827.01 \\
	2000 &       20.23 &     16.77 &     20.61 &    152.90 &     21.58 &   1548.46 &     21.90 &  15577.87 \\
	3000 &      -&- &     39.89 &    210.35 &     45.44 &   2296.17 &     45.56 &  23344.49 \\
	4000 &      -&-&     71.32 &    280.49 &     81.88 &   3079.87 &     83.38 &  31278.97 \\
	5000 &        -&-&    119.82 &    360.98 &    130.84 &   3879.09 &    131.54 &  39249.12 \\
		\hline
	\end{tabular}
	\caption{Running time for settings with reduction method under $\lambda = 0.6$}
	\footnotesize
	\centering
	\label{random15}
	\begin{tabular}{|c|cc|cc|cc|cc|}
		\toprule
			 & \multicolumn{2}{c|}{{$b=10^3$}} & \multicolumn{2}{c|}{{$b=10^4$}} & \multicolumn{2}{c|}{{$b =10^5$}} & \multicolumn{2}{c|}{{$b=10^6$}} \\
		\cmidrule{2-9}
		n& DL-R & DP-R & DL-R & DP-R & DL-R & DP-R & DL-R & DP-R \\
		\midrule
		100 &        0.05 &      0.64 &      0.06 &      6.44 &      0.04 &     64.36 &      0.05 &    651.35 \\
		200 &         0.14 &      1.28 &      0.16 &     13.23 &      0.14 &    129.47 &      0.16 &   1318.83 \\
		500 &      0.83 &      3.13 &      0.76 &     31.86 &      0.87 &    332.51 &      0.84 &   3386.89 \\
		1000 &      3.63 &      6.36 &      2.82 &     65.22 &      3.56 &    681.63 &      3.65 &   7032.14 \\
		2000 &      16.97 &     14.85 &     13.63 &    136.72 &     13.56 &   1371.52 &     15.47 &  14342.41 \\
		3000 &         -&- &     32.35 &    196.39 &     29.26 &   2056.88 &     32.97 &  21602.63 \\
		4000 &       -&-&     60.55 &    262.71 &     54.32 &   2785.99 &     61.34 &  28990.69 \\
		5000 &        -&-&     100.10 &    336.73 &     88.68 &   3522.05 &     96.49 &  36360.91 \\
		\hline
	\end{tabular}
\end{table*}

We also provide visual comparisons using performance curves under the four settings.
Fig. \ref{fig:n1} plots the lifting time against \(b\) with fixed \(n = 1000\).
We can observe that DL and DL-R exhibit negligible growth as \(b\) increases,
whereas DP and DP-R grow almost linearly, with DP exceeding DL-R by more than four orders of magnitude at $b=10^6$. 
Fig. \ref{fig:n} shows running times with respect to $n$ with fixed $b = 10^5$.
DP and DP-R increase nearly linearly with $n$, while DL and DL-R grow at a much slower rate.
DL-R achieves the best overall performance; for $n=5000$, it is nearly two orders of magnitude faster than DP-R. 
These observations align with the pseudo-polynomial time complexity of DP, $\mathcal{O}(nb)$.

\begin{figure}[htbp]
	\centering
	\begin{minipage}[t]{0.45\textwidth}
		\centering
	\includegraphics[width=1.1\textwidth]{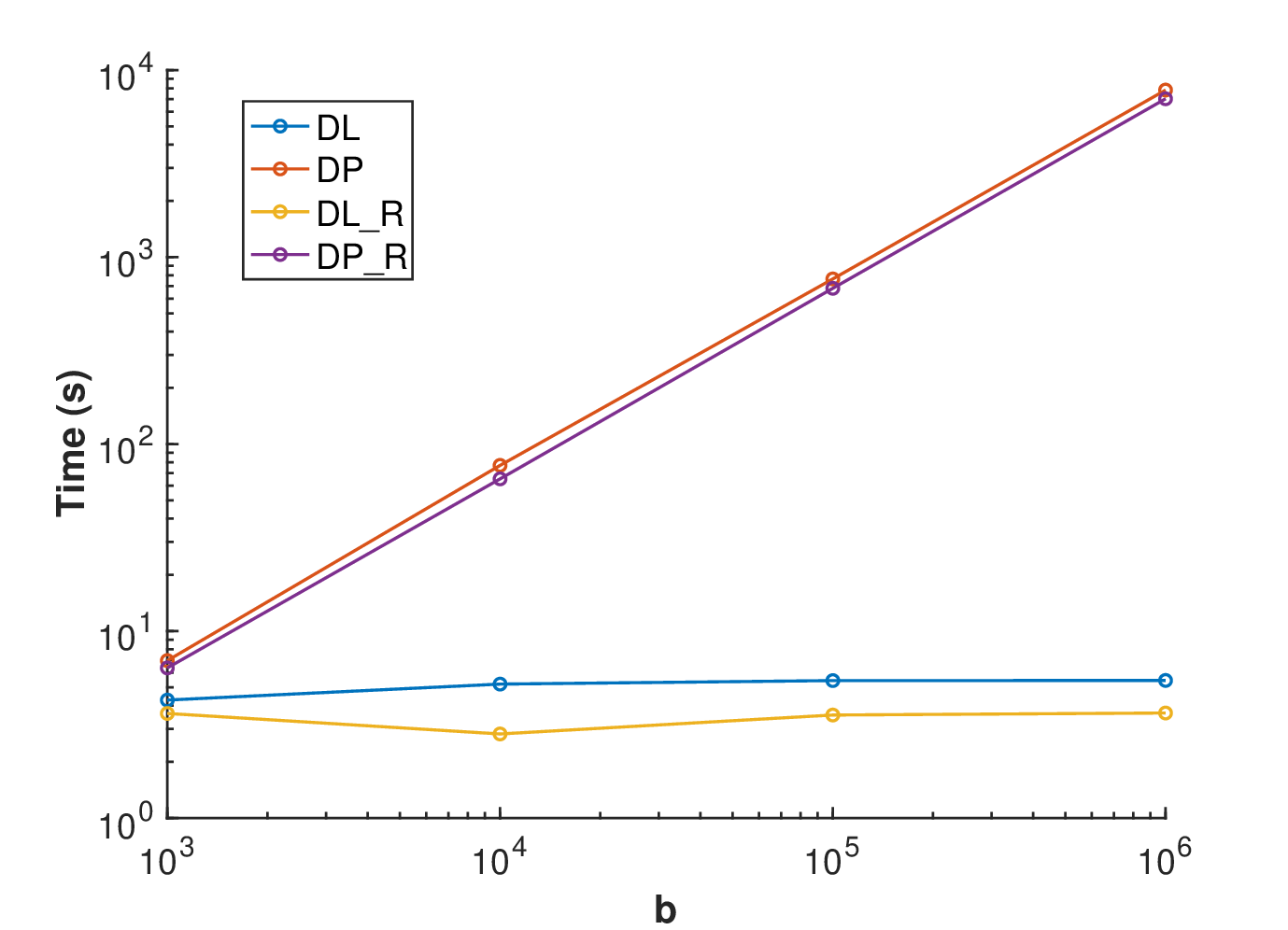}
\caption{Performance of $b$ for four settings under $n=1000,~\lambda = 0.6$}
\label{fig:n1}
	\end{minipage}
	\begin{minipage}[t]{0.45\textwidth}
	\centering
		\includegraphics[width=1.2\textwidth]{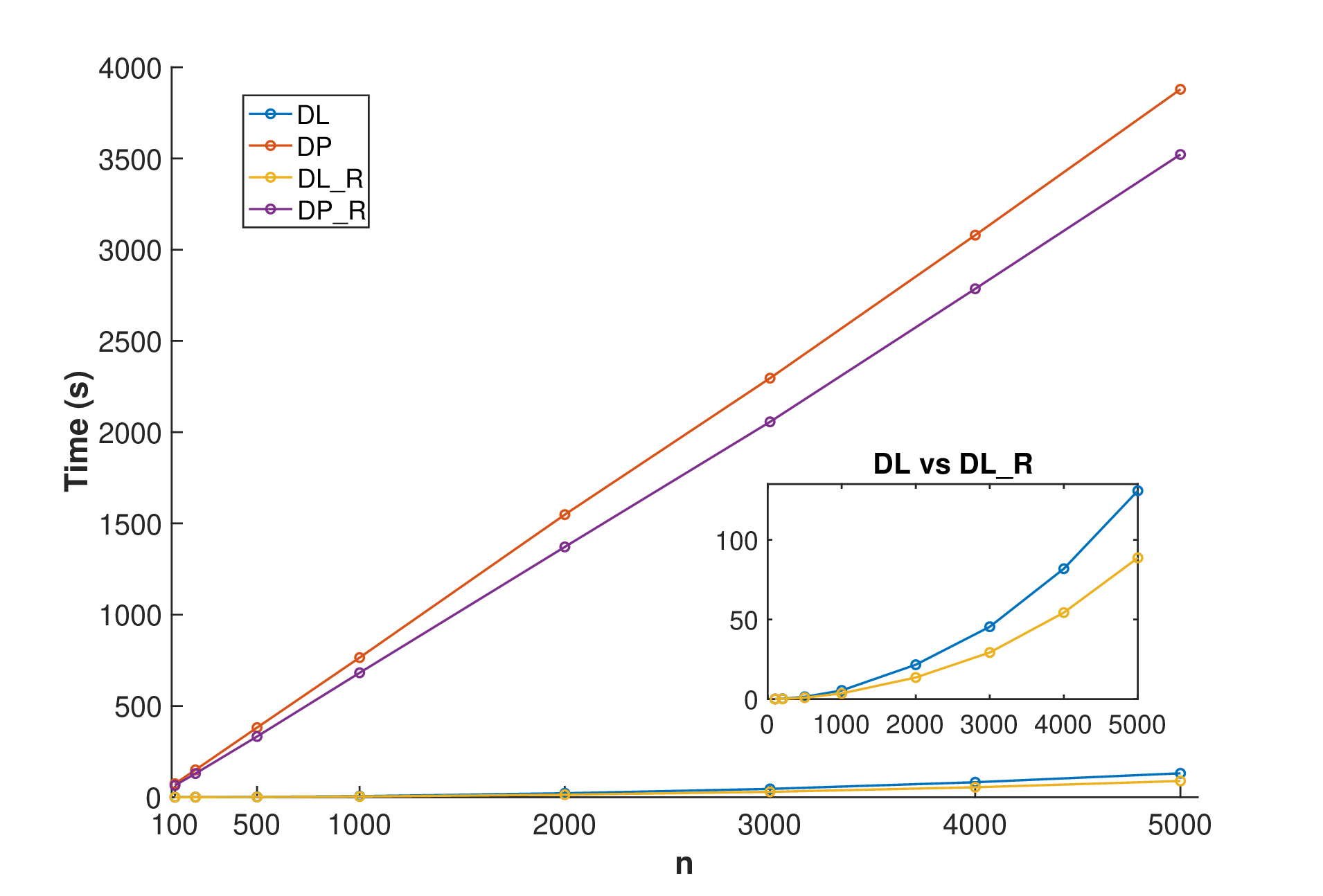}
	\caption{Performance of $n$ for four settings under $b=10^5,~\lambda = 0.6$}
	\label{fig:n}
		\end{minipage}
\end{figure}

\subsubsection{Performance and impact of $\lambda$ and capacity}\label{result1}
In this subsection, we evaluate the impact of the parameter \(\lambda\) on performance.
We report the total lifting time (in seconds) for 1000 knapsack constraints under different combinations of $\lambda$ (= 0.005, 0.01, 0.05, 0.1, 0.2, 0.4, 0.6, 0.8) and capacities ($b = 10^2, 10^3, 10^4, 10^5, 10^6$), with fixed $n = 3000$. 
Each configuration is averaged over five random instances to ensure statistical reliability.

As shown in Table \ref{random6}, DL generally outperforms DP across all values of $\lambda$ and $b$.
For example, when $\lambda=0.2$ and $b=10^6$, DL completes the lifting process in only 5.84\,s, while DP takes 7834.93\,s.
The performance gap between DL and DP widens as $\lambda$ increases, indicating that DL remains more efficient when the capacity becomes larger relative to the total item weight.

Table~\ref{random7} demonstrates that the reduction method significantly improves efficiency when $\lambda$ is large.
When $\lambda \leq 0.4$, the performance difference between DL and DL-R (or DP and DP-R) is relatively small.
However, as $\lambda$ increases, meaning that the knapsack capacity becomes larger relative to the total item weight, the reduction method yields substantial improvements.
This confirms that the reduction method is particularly effective when $\lambda$ takes larger values, as the elimination of redundant dominance relations and DP states becomes increasingly effective.

\begin{table*}[thb]
	\centering
	\caption{Running time for settings without reduction method under $n =3000$}
	\footnotesize
	\label{random6}
	\begin{tabular}{|c|cc|cc|cc|cc|cc|}
		\toprule 
			 & \multicolumn{2}{c|}{{$b=10^2$}}  & \multicolumn{2}{c|}{{$b=10^3$}} & \multicolumn{2}{c|}{{$b=10^4$}} & \multicolumn{2}{c|}{{$b =10^5$}} & \multicolumn{2}{c|}{{$b=10^6$}} \\
		\cmidrule{2-11}	
		$\lambda$ & DL & DP & DL & DP & DL  & DP & DL & DP & DL & DP\\
		\midrule	
		0.005 &      0.40 &      0.64 &      0.34 &      3.14 &      0.29 &     28.85 &      0.36 &    295.28 &      0.37 &   3016.68 \\
		0.01 &      0.54 &      0.75 &      0.33 &      2.17 &      0.37 &     19.05 &      0.33 &    194.51 &      0.33 &   2014.42 \\
		0.05 &      3.40 &      2.00 &      1.42 &      4.32 &      0.77 &     22.11 &      0.72 &    221.87 &      0.76 &   2242.81 \\
		0.1 &      3.32 &      1.99 &      4.03 &      7.61 &      1.89 &     42.71 &      1.69 &    402.27 &      1.77 &   4029.60 \\
		0.2 &        -&- &     10.39 &     12.29 &      6.40 &     83.13 &      5.68 &    778.27 &      5.84 &   7834.93 \\
		0.4 &       -&- &  29.22 &     21.72 &     21.75 &    158.53 &     21.06 &   1543.02 &     21.00 &  15642.21 \\
		0.6 &        -&- &     29.10 &     21.47 &     39.89 &    210.35 &     45.44 &   2296.17 &     45.56 &  23344.49 \\
		0.8 &        -&- &       30.14 &     22.07 &     65.28 &    269.20 &     80.66 &   3032.07 &     82.04 &  30953.53 \\
		\hline
	\end{tabular}
	\caption{Running time for settings with reduction method under $n =3000$}
	\footnotesize
	\label{random7}
	\centering
	\begin{tabular}{|c|cc|cc|cc|cc|cc|}
		\toprule 
		 & \multicolumn{2}{c|}{{$b=10^2$}}  & \multicolumn{2}{c|}{{$b=10^3$}} & \multicolumn{2}{c|}{{$b=10^4$}} & \multicolumn{2}{c|}{{$b =10^5$}} & \multicolumn{2}{c|}{{$b=10^6$}} \\
		\cmidrule{2-11}	
		$\lambda $& DL-R & DP-R & DL-R& DP-R & DL-R & DP-R & DL-R & DP-R & DL-R & DP-R  \\
		\midrule
	0.005 &      0.42 &      0.61 &      0.45 &      3.20 &      0.33 &     28.35 &      0.37 &    292.30 &      0.37 &   2979.94 \\
	0.01 &      0.57 &      0.70 &      0.40 &      2.14 &      0.35 &     18.78 &      0.40 &    192.18 &      0.37 &   2032.87 \\
	0.05 &      4.27 &      1.97 &      1.32 &      4.29 &      0.78 &     22.16 &      0.74 &    221.25 &      0.69 &   2239.75 \\
	0.1 &      3.35 &      1.98 &      4.10 &      7.40 &      1.88 &     42.52 &      1.77 &    399.93 &      1.53 &   4024.52 \\
	0.2 &      -&-&     10.07 &     11.79 &      6.52 &     81.99 &      5.94 &    776.76 &      6.20 &   7824.73 \\
	0.4 &     -&-&      26.24 &     19.86 &     21.05 &    153.85 &     21.61 &   1544.87 &     21.62 &  15621.71 \\
	0.6 &    -&-&     26.61 &     19.65 &     32.35 &    196.39 &     29.26 &   2056.88 &     32.97 &  21602.63 \\
	0.8 &    -&-&    27.59 &     20.13 &     29.38 &    212.84 &     19.47 &   2128.51 &     22.82 &  22396.36 \\
		\hline
	\end{tabular}
\end{table*}

Fig. \ref{fig:lambda1} presents the running time trends as $\lambda$ varies from 0.005 to 0.8 for fixed $b=10^4$. 
Both DL and DL-R exhibit a gradual increase in running time, in contrast, DP and DP-R show much steeper growth, with DP exhibiting almost linear dependence on $\lambda$.

\begin{figure}[htbp]
	\centering
	\begin{minipage}[t]{0.5\textwidth}
		\centering
		\includegraphics[width=1.1\textwidth]{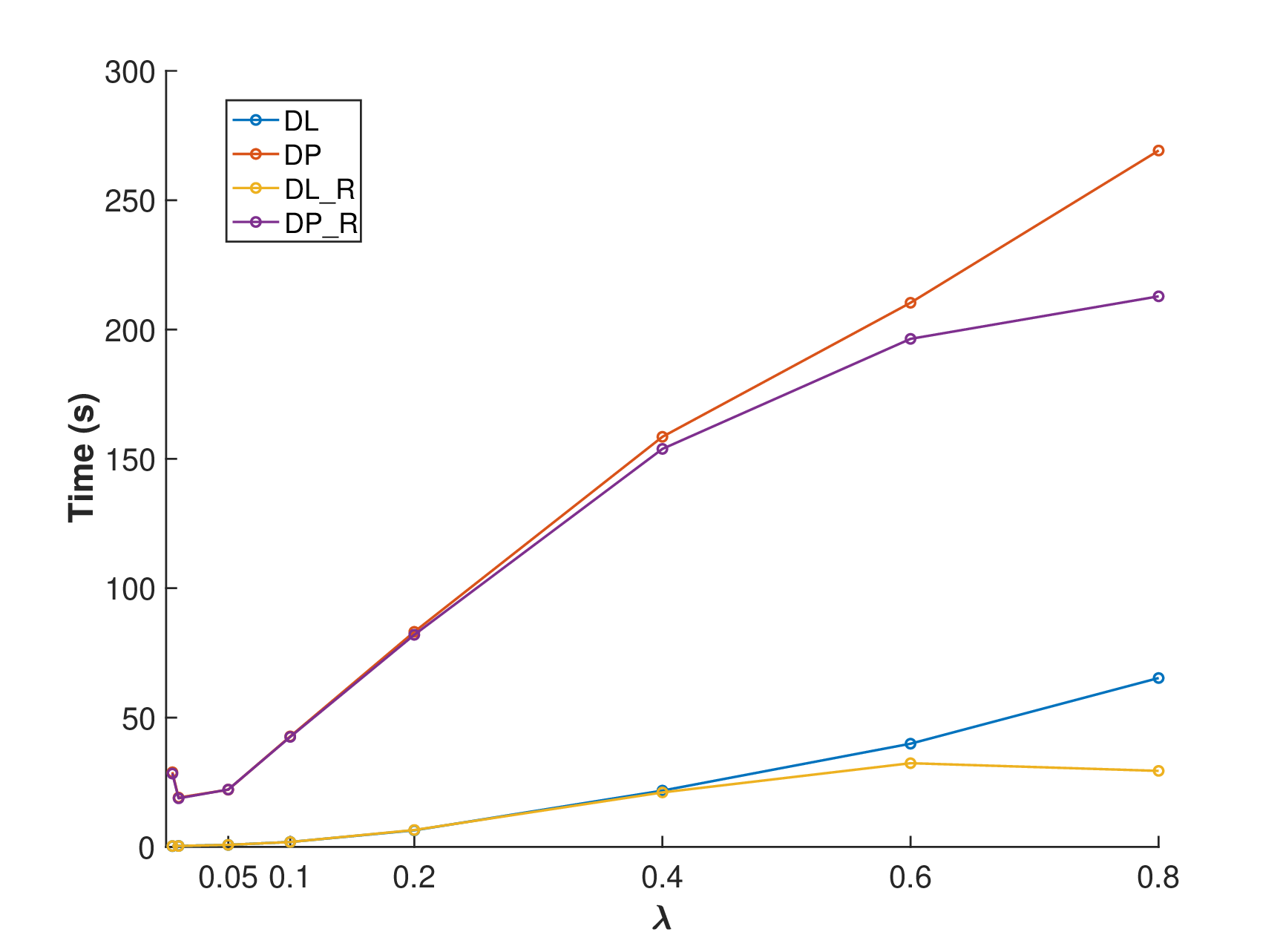}
		\caption{Performance of $\lambda$ for four settings under $n=3000,~b=10^4$}
		\label{fig:lambda1}
	\end{minipage}
	\begin{minipage}[t]{0.45\textwidth}
		\centering
	\end{minipage}
\end{figure}

\subsection{Performance of scale invariance}\label{sec4.2}
To evaluate the scale invariance of our proposed exact sequential lifting algorithm, we implemented it within the exact knapsack separation framework described in \citet{vasilyev2016implementation}.
This framework is a classical implementation of the exact separation algorithm for 0--1 programs with knapsack constraints and consists of four major phases: preprocessing, row generation, numerical correction, and sequential lifting.
The sequential lifting procedure in this framework is implemented by DP.
Therefore, by embedding our proposed algorithm into the lifting procedure of this framework, we are able to directly compare its computational behavior with the DP under identical separation conditions.
We followed the same experimental settings as in \cite{vasilyev2016implementation}, while keeping all other components of the separation framework unchanged:
\begin{itemize}
	\item [1).] The MINKNAP algorithm is used to solve KPs in the row generation procedure.
\item [2).] The lifted cover inequalities are not used.
\item [3).] The variables, which are fixed to one, are down-lifted first, and then the variables, which are fixed to zero, are up-lifted.
\end{itemize}

\subsubsection{Testset}
We adopt the same practical testsets as in \cite{vasilyev2016implementation}, including 
the Generalized Assignment Problem (GAP), 
the Multilevel Generalized Assignment Problem (MGAP),
the Capacitated p-Median Problem (CPMP), and
the Capacitated Network Location Problems (CNLP), to ensure consistency and comparability.

To examine the scale invariance, we introduce a ``Scale" factor that multiplies both the weights and capacity of each constraint by the same factor. 
This operation preserves the combinatorial structure of the problem while increasing the numerical magnitude of the coefficients, 
allowing us to assess the stability and computational robustness of both DL and DP under different scaling levels.
According to the conclusions in Subsection \ref{result1}, the reduction method provides limited effect for instances with small $\lambda$ values. 
Since all instances have $\lambda < 0.05$ in the testsets, the reduction method is not applied in this experiment.
Therefore, the experiment focuses on comparing the performance of the DL and DP.

\subsubsection{Performance}
Table~\ref{scale} summarizes the average lifting times\footnote{Shifted geometric mean, 1\,s for average time} for four practical testsets under different scale factors  (Scale = 1, 10, 100, 1000).
More detailed computational results can be found in Tables \ref{cpmp}, \ref{cnlp}, \ref{mgap}, \ref{gap} in the Appendix \ref{appendix}. 

As shown in Table~\ref{scale}, DL demonstrates excellent scale invariance.
At Scale = 1, DL performs less favorably compared to DP, as the small scale results in very short solving times, where the additional complexity of maintaining the dominance list provides little benefit.
In contrast, the DP method performs efficiently without the need for further optimization.
However, as the scale factor increases, the advantages of the DL method become more apparent, with DL showing more stable performance and significantly better scalability; whereas DP shows a significant increase, especially at Scale = 1000.
For instance, in the GAP testset, the running time of DP increases from 0.28\,s to 59.03\,s as the Scale rises from 1 to 1000, while DL increases only from 0.49\,s to 0.77\,s. 
The result confirms that DL provides more stable computational results when the magnitude of the weights and capacity increases, highlighting its robustness and scale invariance.

\begin{table*}[thb]
	\centering
	\caption{Performance of different scale under DL and DP for four practical testsets}
	\begin{tabular}{|c|cc|cc|cc|cc|}
		\toprule 
		\label{scale}
		& \multicolumn{2}{c|}{{Scale = $1$}} 
		 & \multicolumn{2}{c|}{{Scale = $10$}} & \multicolumn{2}{c|}{{Scale =$100$}} & 
		\multicolumn{2}{c|}{{Scale = $1000$}}  \\
		\cmidrule{2-9}	
		Testset & DL & DP & DL & DP & DL  & DP & DL & DP\\
		\midrule
		CPMP &  0.03 &      0.02&     0.03&      0.02&     0.03 &      0.04&     0.04 &      0.20\\
		CNLP &    0.03 &      0.02&     0.03 &      0.04&     0.04 &      0.18&     0.07 &      1.52\\
		MGAP &0.44 &      0.17&     0.44 &      0.33&     0.49 &      1.46&     0.51 &     12.60\\
		GAP &   0.49 &      0.28&     0.50 &      0.90&     0.52 &      5.53&     0.77 &     59.03\\
		\hline
	\end{tabular}
\end{table*}

\subsection{Comparison of exact lifting and approximate lifting}\label{sec4.3}
Finally, to further evaluate our proposed exact sequential lifting algorithm, we integrate it into the open-source solver HiGHS\footnote{https://github.com/ERGO-Code} by replacing its default approximate lifting algorithm for binary knapsack set.
In HiGHS, the lifting coefficients are generated using superadditive functions, which provide an efficient but approximate approach. 
In contrast, our proposed algorithm performs an exact lifting procedure, without relying on superadditive assumptions.
Specifically, we replace the source code responsible for generating lifting cuts in the binary knapsack set module with our exact algorithm (DL-R), while keeping the original seed inequalities and the default variable order of the lifting procedure to ensure fairness in comparison.

\subsubsection{Testset}
The evaluation is conducted on the MIPLIB 2017\footnote{http://miplib.zib.de} benchmark.
Each instance is tested with five different random seeds and executed with a time limit of 7200 seconds. 

\subsubsection{Performance}
Table \ref{highs} compares the computational results for benchmark MIPLIB 2017 obtained by using the proposed exact lifting and the default approximate lifting implemented in the HiGHS solver.
Detailed computational results can be found in Table \ref{miplib} in the Appendix \ref{appendix}.
For each setting, we report the number of solved instances and the average running time\footnote{Shifted geometric mean, 1\,s for average time}. 
The columns ``Faster" and ``Slower" indicate the number of instances that become at least 10\% faster or slower compared to the default setting. 
Besides, we group the benchmark MIPLIB 2017 into several brackets. 
The bracket ``All" contains the affected instances where the proposed algorithm successfully adds cuts and solved by at least one of the settings.
Each bracket $[n, 7200]$ contains instances that can be solved by the slower setting within at least $n$ seconds. 
Thus, a larger $n$ corresponds to a harder instance. 
The column ``Total'' reports the total number of instances considered in each bracket. 

As it can be observed in Tables \ref{highs}, integrating the proposed exact sequential lifting algorithm into HiGHS yields an overall positive effect. 
In particular, by using the exact separation algorithm, we can solve 4 more instances than those using the default setting on benchmark MIPLIB 2017. 
Specifically, 36 instances are solved faster than that of the default setting, while 25 instances are solved slower. 
Nevertheless, for the solved instances, the average solving time decreases from 849.79\,s to 839.79\,s. 
The improvement becomes more evident for harder instances: 
in the time bracket [1000, 7200], our proposed algorithm solves 70 instances compared to 66 for the default setting, while reducing the average time from 3554.43\,s to 3518.55\,s. 
These results demonstrate that the proposed exact lifting algorithm can effectively enhance the performance of HiGHS in handling binary knapsack cuts, providing stronger cuts without introducing additional computational burden.

\begin{table*}[thb]
	\centering
	\caption{Performance comparison of the exact lifting with the approximate lifting on the binary knapsack set for benchmark MIPLIB 2017}
	\begin{tabular}{|l|c|cc|cc|cc|}
		\toprule 
			\label{highs}
		 & & \multicolumn{2}{c|}{{Exact lifting}}  & \multicolumn{2}{c|}{{Default}}&& \\
		Bracket & Total & Solved & Time & Solved & Time & Faster & Slower\\
		\midrule
				All 	& 	125 & 122 & 839.79 & 118 &  849.79 & 36 & 25\\
		$[10,7200]$ & 121 & 118& 997.20 & 114 &1004.43&32 &25\\
		$[100,7200]$ & 105 &102& 1666.54 & 98 &  1679.87&26&23\\
	$[1000,7200]$ & 73 & 70&  3518.55& 66 &  3554.43 &17&14\\
			\hline
	\end{tabular}
\end{table*}

\section{Conclusions and future works} \label{summary}
In this paper, we propose an efficient exact sequential lifting algorithm for the binary knapsack set. 
By exploiting the dominance list structure, our method eliminates the redundant storage and computation in the DP array, 
enabling efficient handling of instances with non-integer weights and large capacities.
In addition, we introduce a reduction method that optimizes both the dominance list and the DP array, further enhancing computational efficiency.
Experimental results demonstrate that 
1). the proposed algorithm significantly outperforms the DP with arrays in most cases, especially for large problem sizes and capacities.
The reduction method proves especially effective for knapsack capacity becomes large relative to the total item weight.
2). Moreover, the proposed algorithm shows stronger scale invariance than DP, indicating enhanced robustness and reliable performance in practical settings.
3). Finally, the proposed exact algorithm achieves positive effect than the default approximate lifting method, especially for harder instances, confirming its advantage in real-world MIP solving environments.
Overall, the proposed algorithm enables exact sequential lifting for binary knapsack sets with non-integer weights and large capacities, making it suitable for integration into modern MIP solvers.

There still exist some ideas to explore for improving the effectiveness of the lifting process.
Motivated by the complementary strengths of the mentioned different four settings, we aim to develop a hybrid lifting algorithm that dynamically selects between DP, DL  and their reduction variants, adapting to the specific characteristics of each subproblem during the lifting process.
Furthermore, we seek to extend the dominance list structure and reduction methods to more general knapsack sets, such as bounded or unbounded knapsack sets, thereby enhancing the applicability of exact lifting algorithms.

\section*{Acknowledgement}
This work is supported by the National Natural Science Foundations of China (Grant Nos. 12201620, 92473208, 12021001).
The computations were done on the high performance computers of State Key Laboratory of Mathematical Sciences. 

%% file: tex/appendix.tex
\begin{appendices}

\section{Detailed computational results}\label{appendix}
See the Tables \ref{cpmp}, \ref{cnlp}, \ref{mgap}, \ref{gap} and \ref{miplib}.

\begin{center}
	
\begin{longtable}{|c|cc|cc|cc|cc|}
\caption{Performance of different scale under DL and DP for testset CPMP}
		\label{cpmp}\\
		\hline
		& \multicolumn{2}{c|}{{Scale = $1$}}  & \multicolumn{2}{c|}{{Scale = $10$}} & \multicolumn{2}{c|}{{Scale = $100$}} & \multicolumn{2}{c|}{{Scale = $1000$}}  \\
		\cmidrule{2-9}	
		Instance& DL & DP & DL & DP & DL  & DP & DL & DP\\\hline
		\endfirsthead
		\hline
	& \multicolumn{2}{c|}{{Scale = $1$}}  & \multicolumn{2}{c|}{{Scale = $10$}} & \multicolumn{2}{c|}{{Scale = $100$}} & \multicolumn{2}{c|}{{Scale = $1000$}}  \\
	\cmidrule{2-9}	
		Instance& DL & DP & DL & DP & DL  & DP & DL & DP\\\hline
		\endhead
		\hline
		\multicolumn{9}{l}{{Continued on next page}}  
		\endfoot
		 \hline
		\endlastfoot
		ccpx01a      &      0.00 &      0.00 &      0.00 &      0.00 &      0.00 &      0.01 &      0.00 &      0.05 \\
		ccpx01b      &      0.00 &      0.00 &      0.00 &      0.00 &      0.00 &      0.00 &      0.01 &      0.03 \\
		ccpx01c      &      0.01 &      0.00 &      0.00 &      0.00 &      0.00 &      0.00 &      0.02 &      0.02 \\
		ccpx01d      &      0.01 &      0.01 &      0.00 &      0.00 &      0.00 &      0.00 &      0.00 &      0.01 \\
		ccpx02a      &      0.00 &      0.00 &      0.00 &      0.00 &      0.00 &      0.00 &      0.01 &      0.01 \\
		ccpx02b      &      0.00 &      0.00 &      0.00 &      0.00 &      0.00 &      0.01 &      0.00 &      0.01 \\
		ccpx02c      &      0.00 &      0.00 &      0.00 &      0.00 &      0.00 &      0.00 &      0.01 &      0.06 \\
		ccpx02d      &      0.00 &      0.00 &      0.01 &      0.00 &      0.00 &      0.00 &      0.02 &      0.05 \\
		ccpx03a      &      0.00 &      0.00 &      0.00 &      0.00 &      0.00 &      0.00 &      0.01 &      0.02 \\
		ccpx03b      &      0.00 &      0.00 &      0.01 &      0.00 &      0.00 &      0.00 &      0.02 &      0.02 \\
		ccpx03c      &      0.00 &      0.00 &      0.00 &      0.00 &      0.01 &      0.00 &      0.01 &      0.04 \\
		ccpx03d      &      0.01 &      0.00 &      0.01 &      0.00 &      0.01 &      0.00 &      0.02 &      0.04 \\
		ccpx04a      &      0.01 &      0.00 &      0.00 &      0.00 &      0.00 &      0.00 &      0.00 &      0.02 \\
		ccpx04b      &      0.00 &      0.01 &      0.00 &      0.00 &      0.01 &      0.00 &      0.00 &      0.06 \\
		ccpx04c      &      0.01 &      0.00 &      0.03 &      0.00 &      0.01 &      0.00 &      0.01 &      0.05 \\
		ccpx04d      &      0.00 &      0.00 &      0.00 &      0.00 &      0.01 &      0.00 &      0.01 &      0.06 \\
		ccpx05a      &      0.00 &      0.00 &      0.00 &      0.00 &      0.01 &      0.00 &      0.00 &      0.03 \\
		ccpx05b      &      0.01 &      0.00 &      0.03 &      0.00 &      0.00 &      0.00 &      0.00 &      0.06 \\
		ccpx05c      &      0.00 &      0.01 &      0.00 &      0.00 &      0.00 &      0.00 &      0.00 &      0.03 \\
		ccpx05d      &      0.00 &      0.00 &      0.00 &      0.01 &      0.01 &      0.01 &      0.03 &      0.06 \\
		ccpx06a      &      0.00 &      0.00 &      0.00 &      0.00 &      0.00 &      0.01 &      0.00 &      0.00 \\
		ccpx06b      &      0.00 &      0.00 &      0.01 &      0.00 &      0.01 &      0.02 &      0.01 &      0.09 \\
		ccpx06c      &      0.01 &      0.00 &      0.01 &      0.00 &      0.01 &      0.02 &      0.02 &      0.07 \\
		ccpx06d      &      0.00 &      0.01 &      0.00 &      0.00 &      0.00 &      0.01 &      0.02 &      0.03 \\
		ccpx07a      &      0.00 &      0.00 &      0.00 &      0.00 &      0.00 &      0.01 &      0.01 &      0.09 \\
		ccpx07b      &      0.00 &      0.00 &      0.01 &      0.00 &      0.00 &      0.01 &      0.01 &      0.08 \\
		ccpx07c      &      0.01 &      0.00 &      0.01 &      0.01 &      0.01 &      0.01 &      0.01 &      0.08 \\
		ccpx07d      &      0.02 &      0.01 &      0.00 &      0.01 &      0.00 &      0.04 &      0.04 &      0.11 \\
		ccpx08a      &      0.01 &      0.00 &      0.01 &      0.00 &      0.00 &      0.02 &      0.02 &      0.21 \\
		ccpx08b      &      0.00 &      0.00 &      0.01 &      0.00 &      0.01 &      0.03 &      0.03 &      0.10 \\
		ccpx08c      &      0.02 &      0.00 &      0.00 &      0.01 &      0.00 &      0.00 &      0.00 &      0.10 \\
		ccpx08d      &      0.00 &      0.00 &      0.00 &      0.01 &      0.00 &      0.02 &      0.02 &      0.06 \\
		ccpx09a      &      0.00 &      0.00 &      0.00 &      0.00 &      0.01 &      0.02 &      0.03 &      0.07 \\
		ccpx09b      &      0.01 &      0.01 &      0.01 &      0.01 &      0.01 &      0.04 &      0.02 &      0.19 \\
		ccpx09c      &      0.01 &      0.00 &      0.00 &      0.00 &      0.02 &      0.04 &      0.02 &      0.08 \\
		ccpx09d      &      0.01 &      0.00 &      0.01 &      0.02 &      0.01 &      0.00 &      0.02 &      0.08 \\
		ccpx10a      &      0.00 &      0.00 &      0.02 &      0.02 &      0.03 &      0.07 &      0.02 &      0.39 \\
		ccpx10b      &      0.01 &      0.00 &      0.00 &      0.00 &      0.02 &      0.05 &      0.03 &      0.18 \\
		ccpx10c      &      0.01 &      0.00 &      0.02 &      0.01 &      0.02 &      0.00 &      0.02 &      0.18 \\
		ccpx10d      &      0.05 &      0.01 &      0.03 &      0.02 &      0.02 &      0.04 &      0.05 &      0.12 \\
		ccpx11a      &      0.00 &      0.00 &      0.02 &      0.00 &      0.01 &      0.02 &      0.03 &      0.28 \\
		ccpx11b      &      0.02 &      0.00 &      0.02 &      0.00 &      0.02 &      0.02 &      0.00 &      0.15 \\
		ccpx11c      &      0.01 &      0.00 &      0.01 &      0.02 &      0.01 &      0.01 &      0.02 &      0.08 \\
		ccpx11d      &      0.02 &      0.02 &      0.02 &      0.03 &      0.03 &      0.03 &      0.02 &      0.15 \\
		ccpx12a      &      0.00 &      0.00 &      0.00 &      0.00 &      0.00 &      0.03 &      0.00 &      0.09 \\
		ccpx12b      &      0.00 &      0.00 &      0.00 &      0.00 &      0.02 &      0.00 &      0.01 &      0.06 \\
		ccpx12c      &      0.00 &      0.00 &      0.02 &      0.00 &      0.02 &      0.02 &      0.01 &      0.07 \\
		ccpx12d      &      0.01 &      0.00 &      0.02 &      0.00 &      0.02 &      0.02 &      0.03 &      0.14 \\
		ccpx13a      &      0.00 &      0.00 &      0.00 &      0.00 &      0.00 &      0.00 &      0.00 &      0.05 \\
		ccpx13b      &      0.01 &      0.00 &      0.00 &      0.00 &      0.01 &      0.01 &      0.01 &      0.05 \\
		ccpx13c      &      0.02 &      0.00 &      0.01 &      0.00 &      0.03 &      0.01 &      0.02 &      0.05 \\
		ccpx13d      &      0.01 &      0.00 &      0.03 &      0.01 &      0.04 &      0.00 &      0.03 &      0.08 \\
		ccpx14a      &      0.00 &      0.00 &      0.00 &      0.01 &      0.01 &      0.00 &      0.01 &      0.20 \\
		ccpx14b      &      0.02 &      0.00 &      0.04 &      0.01 &      0.02 &      0.00 &      0.01 &      0.07 \\
		ccpx14c      &      0.01 &      0.00 &      0.02 &      0.00 &      0.01 &      0.01 &      0.02 &      0.09 \\
		ccpx14d      &      0.02 &      0.01 &      0.02 &      0.02 &      0.04 &      0.02 &      0.04 &      0.16 \\
		ccpx15a      &      0.01 &      0.00 &      0.01 &      0.01 &      0.01 &      0.07 &      0.04 &      0.43 \\
		ccpx15b      &      0.02 &      0.00 &      0.02 &      0.02 &      0.01 &      0.03 &      0.02 &      0.17 \\
		ccpx15c      &      0.02 &      0.00 &      0.01 &      0.00 &      0.00 &      0.02 &      0.00 &      0.13 \\
		ccpx15d      &      0.05 &      0.01 &      0.04 &      0.00 &      0.04 &      0.03 &      0.04 &      0.15 \\
		ccpx16a      &      0.00 &      0.00 &      0.01 &      0.01 &      0.01 &      0.00 &      0.00 &      0.22 \\
		ccpx16b      &      0.03 &      0.00 &      0.00 &      0.02 &      0.04 &      0.02 &      0.03 &      0.24 \\
		ccpx16c      &      0.04 &      0.01 &      0.02 &      0.00 &      0.03 &      0.00 &      0.02 &      0.13 \\
		ccpx16d      &      0.05 &      0.01 &      0.00 &      0.03 &      0.04 &      0.03 &      0.04 &      0.18 \\
		ccpx17a      &      0.02 &      0.00 &      0.00 &      0.00 &      0.00 &      0.00 &      0.01 &      0.22 \\
		ccpx17b      &      0.02 &      0.01 &      0.00 &      0.01 &      0.00 &      0.06 &      0.02 &      0.15 \\
		ccpx17c      &      0.02 &      0.00 &      0.01 &      0.00 &      0.00 &      0.01 &      0.01 &      0.05 \\
		ccpx17d      &      0.03 &      0.03 &      0.05 &      0.03 &      0.05 &      0.00 &      0.04 &      0.16 \\
		ccpx18a      &      0.00 &      0.00 &      0.02 &      0.01 &      0.01 &      0.04 &      0.01 &      0.25 \\
		ccpx18b      &      0.02 &      0.00 &      0.00 &      0.00 &      0.03 &      0.00 &      0.04 &      0.09 \\
		ccpx18c      &      0.02 &      0.00 &      0.02 &      0.00 &      0.03 &      0.04 &      0.05 &      0.08 \\
		ccpx18d      &      0.03 &      0.00 &      0.03 &      0.05 &      0.08 &      0.02 &      0.06 &      0.18 \\
		ccpx19a      &      0.02 &      0.01 &      0.01 &      0.03 &      0.00 &      0.01 &      0.04 &      0.47 \\
		ccpx19b      &      0.02 &      0.00 &      0.00 &      0.01 &      0.01 &      0.00 &      0.01 &      0.06 \\
		ccpx19c      &      0.01 &      0.00 &      0.02 &      0.01 &      0.01 &      0.02 &      0.01 &      0.05 \\
		ccpx19d      &      0.05 &      0.01 &      0.03 &      0.02 &      0.01 &      0.04 &      0.01 &      0.17 \\
		ccpx20a      &      0.02 &      0.00 &      0.02 &      0.00 &      0.06 &      0.12 &      0.04 &      0.70 \\
		ccpx20b      &      0.02 &      0.02 &      0.02 &      0.03 &      0.06 &      0.05 &      0.04 &      0.50 \\
		ccpx20c      &      0.03 &      0.00 &      0.02 &      0.01 &      0.02 &      0.02 &      0.02 &      0.31 \\
		ccpx20d      &      0.12 &      0.01 &      0.04 &      0.01 &      0.07 &      0.11 &      0.11 &      0.41 \\
		ccpx21a      &      0.02 &      0.02 &      0.04 &      0.03 &      0.02 &      0.03 &      0.00 &      0.36 \\
		ccpx21b      &      0.02 &      0.00 &      0.03 &      0.00 &      0.00 &      0.05 &      0.02 &      0.18 \\
		ccpx21c      &      0.09 &      0.01 &      0.03 &      0.00 &      0.04 &      0.03 &      0.03 &      0.22 \\
		ccpx21d      &      0.08 &      0.03 &      0.06 &      0.02 &      0.06 &      0.05 &      0.07 &      0.28 \\
		ccpx22a      &      0.01 &      0.00 &      0.03 &      0.02 &      0.04 &      0.09 &      0.05 &      1.09 \\
		ccpx22b      &      0.01 &      0.00 &      0.04 &      0.00 &      0.04 &      0.05 &      0.05 &      0.39 \\
		ccpx22c      &      0.01 &      0.02 &      0.07 &      0.02 &      0.04 &      0.08 &      0.04 &      0.32 \\
		ccpx22d      &      0.07 &      0.04 &      0.07 &      0.04 &      0.07 &      0.10 &      0.10 &      0.53 \\
		ccpx23a      &      0.02 &      0.00 &      0.04 &      0.01 &      0.01 &      0.02 &      0.01 &      0.47 \\
		ccpx23b      &      0.00 &      0.00 &      0.07 &      0.01 &      0.00 &      0.05 &      0.06 &      0.52 \\
		ccpx23c      &      0.05 &      0.03 &      0.07 &      0.03 &      0.05 &      0.07 &      0.05 &      0.46 \\
		ccpx23d      &      0.12 &      0.03 &      0.11 &      0.03 &      0.07 &      0.13 &      0.12 &      0.52 \\
		ccpx24a      &      0.01 &      0.01 &      0.01 &      0.02 &      0.03 &      0.09 &      0.02 &      0.79 \\
		ccpx24b      &      0.03 &      0.00 &      0.01 &      0.00 &      0.01 &      0.04 &      0.05 &      0.18 \\
		ccpx24c      &      0.03 &      0.01 &      0.01 &      0.01 &      0.00 &      0.02 &      0.04 &      0.16 \\
		ccpx24d      &      0.04 &      0.02 &      0.06 &      0.04 &      0.05 &      0.07 &      0.10 &      0.34 \\
		ccpx25a      &      0.03 &      0.00 &      0.00 &      0.01 &      0.02 &      0.03 &      0.02 &      0.21 \\
		ccpx25b      &      0.02 &      0.00 &      0.00 &      0.01 &      0.00 &      0.02 &      0.01 &      0.08 \\
		ccpx25c      &      0.02 &      0.03 &      0.01 &      0.00 &      0.01 &      0.03 &      0.02 &      0.10 \\
		ccpx25d      &      0.01 &      0.01 &      0.02 &      0.00 &      0.02 &      0.01 &      0.03 &      0.09 \\
		ccpx26a      &      0.01 &      0.01 &      0.02 &      0.00 &      0.00 &      0.05 &      0.02 &      0.50 \\
		ccpx26b      &      0.02 &      0.00 &      0.02 &      0.01 &      0.02 &      0.03 &      0.02 &      0.14 \\
		ccpx26c      &      0.00 &      0.00 &      0.00 &      0.00 &      0.02 &      0.02 &      0.02 &      0.06 \\
		ccpx26d      &      0.04 &      0.01 &      0.03 &      0.01 &      0.03 &      0.05 &      0.03 &      0.15 \\
		ccpx27a      &      0.07 &      0.01 &      0.04 &      0.07 &      0.07 &      0.15 &      0.06 &      1.86 \\
		ccpx27b      &      0.07 &      0.01 &      0.05 &      0.02 &      0.07 &      0.08 &      0.07 &      0.58 \\
		ccpx27c      &      0.03 &      0.02 &      0.06 &      0.01 &      0.05 &      0.06 &      0.03 &      0.41 \\
		ccpx27d      &      0.14 &      0.09 &      0.15 &      0.03 &      0.18 &      0.13 &      0.17 &      0.89 \\
		ccpx28a      &      0.01 &      0.01 &      0.01 &      0.00 &      0.00 &      0.04 &      0.02 &      0.22 \\
		ccpx28b      &      0.00 &      0.01 &      0.01 &      0.00 &      0.02 &      0.01 &      0.01 &      0.12 \\
		ccpx28c      &      0.01 &      0.01 &      0.03 &      0.00 &      0.01 &      0.03 &      0.00 &      0.07 \\
		ccpx28d      &      0.03 &      0.01 &      0.07 &      0.02 &      0.06 &      0.01 &      0.02 &      0.11 \\
		ccpx29a      &      0.02 &      0.00 &      0.00 &      0.00 &      0.02 &      0.00 &      0.01 &      0.10 \\
		ccpx29b      &      0.02 &      0.02 &      0.02 &      0.01 &      0.04 &      0.00 &      0.02 &      0.12 \\
		ccpx29c      &      0.03 &      0.00 &      0.01 &      0.00 &      0.04 &      0.02 &      0.03 &      0.11 \\
		ccpx29d      &      0.04 &      0.01 &      0.01 &      0.02 &      0.04 &      0.03 &      0.03 &      0.12 \\
		ccpx30a      &      0.01 &      0.00 &      0.01 &      0.01 &      0.03 &      0.02 &      0.01 &      0.22 \\
		ccpx30b      &      0.02 &      0.00 &      0.00 &      0.00 &      0.00 &      0.01 &      0.04 &      0.15 \\
		ccpx30c      &      0.01 &      0.02 &      0.02 &      0.00 &      0.02 &      0.00 &      0.05 &      0.16 \\
		ccpx30d      &      0.03 &      0.00 &      0.06 &      0.00 &      0.04 &      0.04 &      0.03 &      0.18 \\
		ccpx31a      &      0.00 &      0.00 &      0.01 &      0.00 &      0.02 &      0.07 &      0.03 &      0.47 \\
		ccpx31b      &      0.05 &      0.01 &      0.04 &      0.00 &      0.04 &      0.06 &      0.04 &      0.35 \\
		ccpx31c      &      0.04 &      0.01 &      0.02 &      0.00 &      0.02 &      0.02 &      0.02 &      0.19 \\
		ccpx31d      &      0.08 &      0.02 &      0.06 &      0.02 &      0.06 &      0.02 &      0.05 &      0.21 \\
		ccpx32a      &      0.05 &      0.02 &      0.02 &      0.05 &      0.05 &      0.24 &      0.06 &      2.21 \\
		ccpx32b      &      0.15 &      0.07 &      0.11 &      0.07 &      0.24 &      0.36 &      0.24 &      2.51 \\
		ccpx32c      &      0.28 &      0.06 &      0.16 &      0.06 &      0.16 &      0.26 &      0.14 &      1.77 \\
		ccpx32d      &      0.32 &      0.13 &      0.31 &      0.13 &      0.44 &      0.28 &      0.41 &      2.42 \\
		ccpx33a      &      0.03 &      0.00 &      0.04 &      0.01 &      0.02 &      0.07 &      0.02 &      0.56 \\
		ccpx33b      &      0.03 &      0.00 &      0.03 &      0.01 &      0.05 &      0.06 &      0.04 &      0.33 \\
		ccpx33c      &      0.04 &      0.00 &      0.07 &      0.01 &      0.08 &      0.06 &      0.10 &      0.51 \\
		ccpx33d      &      0.09 &      0.03 &      0.09 &      0.01 &      0.07 &      0.06 &      0.10 &      0.41 \\
		ccpx34a      &      0.04 &      0.00 &      0.02 &      0.00 &      0.03 &      0.12 &      0.01 &      0.81 \\
		ccpx34b      &      0.12 &      0.04 &      0.08 &      0.03 &      0.10 &      0.23 &      0.13 &      1.40 \\
		ccpx34d      &      0.30 &      0.17 &      0.33 &      0.10 &      0.32 &      0.20 &      0.37 &      2.04 \\
		ccpx35a      &      0.06 &      0.03 &      0.03 &      0.02 &      0.01 &      0.19 &      0.05 &      1.38 \\
		ccpx35c      &      0.05 &      0.01 &      0.08 &      0.01 &      0.06 &      0.07 &      0.09 &      0.28 \\
		ccpx35d      &      0.05 &      0.01 &      0.07 &      0.02 &      0.07 &      0.04 &      0.11 &      0.37 \\
		ccpx36a      &      0.02 &      0.00 &      0.03 &      0.01 &      0.00 &      0.06 &      0.03 &      0.56 \\
		ccpx36b      &      0.01 &      0.01 &      0.03 &      0.03 &      0.05 &      0.06 &      0.03 &      0.36 \\
		ccpx36c      &      0.04 &      0.01 &      0.03 &      0.02 &      0.04 &      0.05 &      0.02 &      0.32 \\
		ccpx36d      &      0.07 &      0.03 &      0.08 &      0.00 &      0.08 &      0.07 &      0.08 &      0.40 \\
		ccpx37a      &      0.01 &      0.01 &      0.01 &      0.01 &      0.00 &      0.05 &      0.00 &      0.35 \\
		ccpx37b      &      0.02 &      0.00 &      0.04 &      0.01 &      0.02 &      0.01 &      0.02 &      0.10 \\
		ccpx37c      &      0.07 &      0.01 &      0.02 &      0.01 &      0.04 &      0.04 &      0.04 &      0.19 \\
		ccpx37d      &      0.08 &      0.00 &      0.05 &      0.05 &      0.04 &      0.06 &      0.07 &      0.34 \\
		ccpx38a      &      0.01 &      0.00 &      0.04 &      0.02 &      0.01 &      0.07 &      0.06 &      0.66 \\
		ccpx38c      &      0.06 &      0.02 &      0.06 &      0.02 &      0.06 &      0.02 &      0.04 &      0.34 \\
		ccpx38d      &      0.13 &      0.03 &      0.15 &      0.05 &      0.15 &      0.11 &      0.16 &      0.83 \\
		ccpx39a      &      0.00 &      0.00 &      0.01 &      0.00 &      0.00 &      0.08 &      0.00 &      0.37 \\
		ccpx39b      &      0.04 &      0.00 &      0.00 &      0.00 &      0.04 &      0.05 &      0.03 &      0.28 \\
		ccpx39c      &      0.04 &      0.00 &      0.03 &      0.01 &      0.02 &      0.03 &      0.01 &      0.17 \\
		ccpx39d      &      0.07 &      0.00 &      0.04 &      0.01 &      0.02 &      0.04 &      0.11 &      0.22 \\
		ccpx40a      &      0.04 &      0.01 &      0.01 &      0.01 &      0.06 &      0.08 &      0.02 &      0.61 \\
		ccpx40b      &      0.06 &      0.00 &      0.06 &      0.00 &      0.01 &      0.08 &      0.03 &      0.68 \\
		ccpx40c      &      0.05 &      0.02 &      0.06 &      0.01 &      0.03 &      0.05 &      0.07 &      0.48 \\
		dd1          &      0.02 &      0.01 &      0.01 &      0.03 &      0.02 &      0.09 &      0.15 &      1.72 \\
		dd2          &      0.01 &      0.00 &      0.00 &      0.01 &      0.02 &      0.08 &      0.08 &      1.20 \\
		dd3a         &      0.04 &      0.00 &      0.01 &      0.01 &      0.05 &      0.09 &      0.21 &      1.59 \\
		dd3b         &      0.01 &      0.00 &      0.01 &      0.00 &      0.01 &      0.00 &      0.02 &      0.28 \\
		dd4a         &      0.06 &      0.01 &      0.03 &      0.02 &      0.07 &      0.25 &      0.27 &      6.42 \\
		dd4b         &      0.02 &      0.00 &      0.02 &      0.00 &      0.01 &      0.04 &      0.05 &      0.65 \\
		\hline
\end{longtable}
\end{center}

\begin{center}
\begin{longtable}{|c|cc|cc|cc|cc|}
	\caption{Performance of different scale under DL and DP for testset CNLP}
	\label{cnlp}\\	
	\hline
		& \multicolumn{2}{c|}{{Scale = $1$}}  & \multicolumn{2}{c|}{{Scale = $10$}} & \multicolumn{2}{c|}{{Scale = $100$}} & \multicolumn{2}{c|}{{Scale = $1000$}}  \\
	\cmidrule{2-9}	
	Instance& DL & DP & DL & DP & DL  & DP & DL & DP\\\hline
	\endfirsthead
	\hline
		& \multicolumn{2}{c|}{{Scale = $1$}}  & \multicolumn{2}{c|}{{Scale = $10$}} & \multicolumn{2}{c|}{{Scale = $100$}} & \multicolumn{2}{c|}{{Scale = $1000$}}  \\
	\cmidrule{2-9}	
	Instance& DL & DP & DL & DP & DL  & DP & DL & DP\\\hline
	\endhead
	\hline
	\multicolumn{9}{l}{{Continued on next page}}  
	\endfoot
	\hline
	\endlastfoot
		p18-dc       &      0.00 &      0.00 &      0.00 &      0.00 &      0.00 &      0.00 &      0.00 &      0.00 \\
		p18-d        &      0.02 &      0.02 &      0.02 &      0.00 &      0.01 &      0.11 &      0.06 &      0.77 \\
		p20-a        &      0.02 &      0.01 &      0.01 &      0.01 &      0.01 &      0.05 &      0.03 &      0.48 \\
		p20-b        &      0.02 &      0.01 &      0.02 &      0.00 &      0.01 &      0.02 &      0.01 &      0.38 \\
		p20-d        &      0.02 &      0.01 &      0.02 &      0.02 &      0.00 &      0.08 &      0.02 &      0.38 \\
		p21-a        &      0.03 &      0.00 &      0.02 &      0.04 &      0.04 &      0.06 &      0.05 &      0.38 \\
		p21-b        &      0.00 &      0.00 &      0.03 &      0.00 &      0.05 &      0.04 &      0.03 &      0.32 \\
		p26-a        &      0.02 &      0.01 &      0.02 &      0.01 &      0.02 &      0.16 &      0.04 &      0.84 \\
		p26-b        &      0.02 &      0.01 &      0.00 &      0.01 &      0.02 &      0.09 &      0.03 &      0.64 \\
		p30-a        &      0.03 &      0.01 &      0.03 &      0.02 &      0.07 &      0.19 &      0.09 &      1.10 \\
		p30-b        &      0.03 &      0.00 &      0.03 &      0.01 &      0.02 &      0.07 &      0.04 &      0.81 \\
		p31-a        &      0.00 &      0.00 &      0.01 &      0.02 &      0.00 &      0.11 &      0.06 &      0.90 \\
		p33-a        &      0.03 &      0.05 &      0.07 &      0.00 &      0.05 &      0.06 &      0.04 &      0.98 \\
		p33-b        &      0.02 &      0.02 &      0.04 &      0.02 &      0.03 &      0.08 &      0.05 &      0.80 \\
		p33-d        &      0.00 &      0.00 &      0.01 &      0.01 &      0.01 &      0.06 &      0.00 &      0.45 \\
		p34-d        &      0.02 &      0.02 &      0.00 &      0.03 &      0.01 &      0.09 &      0.05 &      1.25 \\
		p35-d        &      0.03 &      0.00 &      0.02 &      0.02 &      0.03 &      0.24 &      0.06 &      2.44 \\
		p36-a        &      0.02 &      0.04 &      0.05 &      0.03 &      0.06 &      0.28 &      0.10 &      2.62 \\
		p36-b        &      0.01 &      0.01 &      0.02 &      0.03 &      0.03 &      0.19 &      0.10 &      2.23 \\
		p36-d        &      0.01 &      0.01 &      0.00 &      0.00 &      0.00 &      0.12 &      0.02 &      0.73 \\
		p37-a        &      0.02 &      0.00 &      0.02 &      0.04 &      0.02 &      0.29 &      0.10 &      2.80 \\
		p37-b        &      0.02 &      0.02 &      0.00 &      0.02 &      0.02 &      0.25 &      0.12 &      2.61 \\
		p37-d        &      0.01 &      0.01 &      0.01 &      0.00 &      0.00 &      0.09 &      0.01 &      1.03 \\
		p38-d        &      0.00 &      0.01 &      0.02 &      0.00 &      0.00 &      0.05 &      0.01 &      0.47 \\
		p39-a        &      0.00 &      0.01 &      0.02 &      0.02 &      0.03 &      0.19 &      0.03 &      1.02 \\
		p39-b        &      0.00 &      0.00 &      0.07 &      0.02 &      0.01 &      0.06 &      0.03 &      1.13 \\
		p39-d        &      0.01 &      0.00 &      0.03 &      0.01 &      0.01 &      0.09 &      0.01 &      0.72 \\
		p40-a        &      0.03 &      0.02 &      0.03 &      0.04 &      0.05 &      0.24 &      0.13 &      1.86 \\
		p40-b        &      0.03 &      0.03 &      0.02 &      0.01 &      0.02 &      0.20 &      0.08 &      2.17 \\
		p40-d        &      0.02 &      0.01 &      0.02 &      0.02 &      0.04 &      0.11 &      0.05 &      0.85 \\
		p41-d        &      0.01 &      0.03 &      0.01 &      0.04 &      0.03 &      0.47 &      0.12 &      5.09 \\
		p42-d        &      0.02 &      0.01 &      0.06 &      0.05 &      0.01 &      0.21 &      0.04 &      2.12 \\
		p43-d        &      0.02 &      0.00 &      0.00 &      0.01 &      0.02 &      0.22 &      0.10 &      2.57 \\
		p44-a        &      0.03 &      0.06 &      0.02 &      0.07 &      0.14 &      0.35 &      0.19 &      3.58 \\
		p44-b        &      0.07 &      0.04 &      0.09 &      0.11 &      0.13 &      0.44 &      0.19 &      4.05 \\
		p44-d        &      0.03 &      0.01 &      0.01 &      0.02 &      0.03 &      0.10 &      0.04 &      0.74 \\
		p45-d        &      0.02 &      0.00 &      0.04 &      0.08 &      0.00 &      0.21 &      0.07 &      2.32 \\
		p46-a        &      0.05 &      0.02 &      0.03 &      0.07 &      0.04 &      0.48 &      0.14 &      4.11 \\
		p46-b        &      0.08 &      0.01 &      0.05 &      0.07 &      0.01 &      0.37 &      0.09 &      3.41 \\
		p47-a        &      0.05 &      0.02 &      0.08 &      0.13 &      0.07 &      0.41 &      0.15 &      3.67 \\
		p47-b        &      0.05 &      0.03 &      0.04 &      0.06 &      0.03 &      0.35 &      0.10 &      3.40 \\
		p47-d        &      0.06 &      0.00 &      0.07 &      0.05 &      0.04 &      0.28 &      0.11 &      2.74 \\
		p48-d        &      0.02 &      0.00 &      0.00 &      0.01 &      0.02 &      0.27 &      0.05 &      3.54 \\
		p49-a        &      0.10 &      0.02 &      0.02 &      0.06 &      0.08 &      0.29 &      0.10 &      3.85 \\
		p49-b        &      0.03 &      0.04 &      0.02 &      0.07 &      0.10 &      0.33 &      0.15 &      3.55 \\
		p49-d        &      0.00 &      0.01 &      0.01 &      0.06 &      0.02 &      0.17 &      0.09 &      2.34 \\
		p51-a        &      0.05 &      0.02 &      0.06 &      0.07 &      0.06 &      0.53 &      0.16 &      5.27 \\
		p52-a        &      0.01 &      0.03 &      0.08 &      0.04 &      0.03 &      0.25 &      0.12 &      3.01 \\
		p52-b        &      0.04 &      0.01 &      0.09 &      0.04 &      0.02 &      0.29 &      0.14 &      3.06 \\
		p53-a        &      0.02 &      0.04 &      0.04 &      0.05 &      0.06 &      0.40 &      0.11 &      3.78 \\
		p53-b        &      0.04 &      0.02 &      0.02 &      0.03 &      0.04 &      0.30 &      0.15 &      2.93 \\
		p54-b        &      0.02 &      0.01 &      0.03 &      0.05 &      0.05 &      0.35 &      0.05 &      3.72 \\
		p55-a        &      0.10 &      0.06 &      0.10 &      0.19 &      0.08 &      0.61 &      0.18 &      5.54 \\
		p55-d        &      0.07 &      0.00 &      0.03 &      0.05 &      0.06 &      0.31 &      0.11 &      2.83 \\
		p57-d        &      0.07 &      0.02 &      0.01 &      0.04 &      0.04 &      0.36 &      0.08 &      4.64 \\
		\hline
\end{longtable}
\end{center}

\begin{center}
\begin{longtable}{|c|cc|cc|cc|cc|}
	\caption{Performance of different scale under DL and DP for testset MGAP}
	\label{mgap}\\	
	\hline
		& \multicolumn{2}{c|}{{Scale = $1$}}  & \multicolumn{2}{c|}{{Scale = $10$}} & \multicolumn{2}{c|}{{Scale = $100$}} & \multicolumn{2}{c|}{{Scale = $1000$}}  \\
	\cmidrule{2-9}	
	Instance& DL & DP & DL & DP & DL  & DP & DL & DP\\\hline
	\endfirsthead
	\hline
		& \multicolumn{2}{c|}{{Scale = $1$}}  & \multicolumn{2}{c|}{{Scale = $10$}} & \multicolumn{2}{c|}{{Scale = $100$}} & \multicolumn{2}{c|}{{Scale = $1000$}}  \\
	\cmidrule{2-9}	
	Instance& DL & DP & DL & DP & DL  & DP & DL & DP\\\hline
	\endhead
	\hline
	\multicolumn{9}{l}{{Continued on next page}}  
	\endfoot
	\hline
	\endlastfoot	
		c10400       &      0.16 &      0.13 &      0.17 &      0.28 &      0.10 &      1.82 &      0.29 &     18.80 \\
		c15900       &      0.50 &      0.24 &      0.57 &      0.83 &      0.50 &      5.55 &      0.62 &     55.33 \\
		c20400       &      0.23 &      0.09 &      0.16 &      0.29 &      0.24 &      1.14 &      0.29 &     11.69 \\
			c201600      &      1.09 &      0.40 &      1.11 &      1.40 &      1.13 &      9.47 &      1.41 &     96.45 \\
		c30900       &      0.56 &      0.26 &      0.53 &      0.51 &      0.50 &      3.19 &      0.66 &     28.77 \\
		c401600      &      0.17 &      0.04 &      0.16 &      0.15 &      0.21 &      0.72 &      0.22 &      6.69 \\
		c60900       &      0.21 &      0.06 &      0.22 &      0.15 &      0.20 &      0.44 &      0.22 &      3.56 \\
		c801600      &      0.27 &      0.08 &      0.29 &      0.14 &      0.25 &      0.55 &      0.22 &      4.60 \\
		E1010003-5   &      0.12 &      0.03 &      0.13 &      0.10 &      0.12 &      0.50 &      0.18 &      4.38 \\
		E1010005-3   &      0.33 &      0.12 &      0.23 &      0.34 &      0.29 &      1.00 &      0.38 &      9.19 \\
		E1010005-4   &      0.27 &      0.10 &      0.23 &      0.19 &      0.28 &      0.78 &      0.26 &      7.03 \\
		e10400       &      0.20 &      0.08 &      0.18 &      0.20 &      0.11 &      1.07 &      0.20 &     10.60 \\
		E1520004-1   &      0.46 &      0.24 &      0.38 &      0.36 &      0.49 &      1.74 &      0.45 &     16.57 \\
		E1520004-2   &      0.35 &      0.12 &      0.30 &      0.23 &      0.32 &      1.31 &      0.37 &     11.57 \\
		E1520004-3   &      0.36 &      0.13 &      0.46 &      0.33 &      0.33 &      1.68 &      0.34 &     14.29 \\
		E1520004-4   &      0.34 &      0.17 &      0.36 &      0.39 &      0.37 &      1.74 &      0.35 &     15.41 \\
		E1520004-5   &      0.45 &      0.20 &      0.35 &      0.40 &      0.25 &      1.94 &      0.48 &     16.57 \\
		E1520005-2   &      0.54 &      0.25 &      0.49 &      0.47 &      0.53 &      2.54 &      0.54 &     23.51 \\
		E1520005-3   &      0.68 &      0.36 &      0.88 &      0.67 &      0.95 &      4.25 &      0.85 &     38.25 \\
		E1520005-5   &      0.41 &      0.20 &      0.60 &      0.48 &      0.59 &      2.47 &      0.53 &     21.62 \\
		e15900       &      0.82 &      0.32 &      0.83 &      1.07 &      0.79 &      6.15 &      0.91 &     58.22 \\
		E2010003-3   &      0.27 &      0.13 &      0.15 &      0.18 &      0.28 &      0.56 &      0.33 &      4.53 \\
		E2010003-5   &      0.19 &      0.05 &      0.20 &      0.17 &      0.23 &      0.52 &      0.30 &      3.68 \\
		E2010004-1   &      0.45 &      0.15 &      0.41 &      0.22 &      0.36 &      1.01 &      0.41 &      9.53 \\
		E2010004-3   &      0.41 &      0.14 &      0.49 &      0.17 &      0.32 &      0.85 &      0.50 &      6.90 \\
		E2010004-5   &      0.34 &      0.11 &      0.21 &      0.18 &      0.21 &      0.65 &      0.26 &      5.43 \\
		E2010005-2   &      0.48 &      0.21 &      0.48 &      0.25 &      0.53 &      1.08 &      0.60 &      9.65 \\
		E2010005-4   &      0.46 &      0.18 &      0.40 &      0.27 &      0.41 &      0.90 &      0.42 &      8.07 \\
		e201600      &      1.97 &      0.73 &      1.97 &      2.59 &      2.22 &     14.62 &      2.33 &    144.80 \\
		E3010003-1   &      0.26 &      0.07 &      0.23 &      0.12 &      0.30 &      0.46 &      0.35 &      4.49 \\
		E3010003-3   &      0.27 &      0.07 &      0.29 &      0.14 &      0.42 &      0.71 &      0.35 &      4.98 \\
		E3010003-5   &      0.27 &      0.09 &      0.43 &      0.23 &      0.34 &      0.85 &      0.40 &      6.30 \\
		E3010004-2   &      0.36 &      0.13 &      0.40 &      0.22 &      0.38 &      0.93 &      0.41 &      7.15 \\
		E3010004-4   &      0.48 &      0.15 &      0.44 &      0.19 &      0.45 &      0.68 &      0.56 &      6.71 \\
		E3010005-1   &      0.34 &      0.17 &      0.62 &      0.29 &      0.69 &      1.04 &      0.69 &     10.04 \\
		E3010005-3   &      0.47 &      0.13 &      0.47 &      0.21 &      0.54 &      0.92 &      0.62 &      7.47 \\
		E3010005-5   &      0.56 &      0.22 &      0.55 &      0.25 &      0.54 &      1.20 &      0.66 &      9.60 \\
		E3020004-2   &      0.69 &      0.24 &      0.61 &      0.40 &      0.70 &      1.93 &      0.79 &     16.61 \\
		E3020004-4   &      0.77 &      0.29 &      0.82 &      0.66 &      0.89 &      2.20 &      0.85 &     19.97 \\
		E3020005-2   &      0.90 &      0.30 &      1.08 &      0.53 &      1.07 &      2.81 &      1.01 &     22.81 \\
		E3020005-3   &      1.17 &      0.42 &      1.24 &      0.79 &      1.31 &      3.29 &      1.52 &     27.87 \\
		E3020005-4   &      0.80 &      0.28 &      1.03 &      0.57 &      0.83 &      2.49 &      0.96 &     20.47 \\
		E3020005-5   &      1.19 &      0.40 &      1.27 &      0.62 &      1.16 &      3.34 &      1.20 &     29.41 \\
		e30900       &      1.35 &      0.57 &      1.48 &      1.10 &      1.41 &      5.01 &      1.56 &     41.10 \\
		e60900       &      0.70 &      0.17 &      0.57 &      0.28 &      0.63 &      1.02 &      0.69 &      7.85 \\
		e801600      &      0.33 &      0.11 &      0.32 &      0.11 &      0.35 &      0.49 &      0.34 &      3.40 \\
		\hline
\end{longtable}
\end{center}

\begin{center}
	\begin{longtable}{|c|cc|cc|cc|cc|}
		\caption{Performance of different scale under DL and DP for testset GAP}
		\label{gap}\\	
		\hline
			& \multicolumn{2}{c|}{{Scale = $1$}}  & \multicolumn{2}{c|}{{Scale = $10$}} & \multicolumn{2}{c|}{{Scale = $100$}} & \multicolumn{2}{c|}{{Scale = $1000$}}  \\
		\cmidrule{2-9}	
		Instance& DL & DP & DL & DP & DL  & DP & DL & DP\\\hline
		\endfirsthead
		\hline
			& \multicolumn{2}{c|}{{Scale = $1$}}  & \multicolumn{2}{c|}{{Scale = $10$}} & \multicolumn{2}{c|}{{Scale = $100$}} & \multicolumn{2}{c|}{{Scale = $1000$}}  \\
		\cmidrule{2-9}	
		Instance& DL & DP & DL & DP & DL  & DP & DL & DP\\\hline
		\endhead
		\hline
		\multicolumn{9}{l}{{Continued on next page}}  
		\endfoot
		\hline
		\endlastfoot	
		c05100       &      0.02 &      0.01 &      0.01 &      0.00 &      0.02 &      0.25 &      0.03 &      2.22 \\
		c05200       &      0.08 &      0.03 &      0.11 &      0.18 &      0.03 &      1.51 &      0.23 &     16.20 \\
		c10100       &      0.01 &      0.00 &      0.03 &      0.00 &      0.01 &      0.12 &      0.04 &      1.07 \\
		c10200       &      0.07 &      0.04 &      0.10 &      0.16 &      0.07 &      0.81 &      0.11 &      8.31 \\
		c10400       &      0.22 &      0.18 &      0.22 &      0.73 &      0.26 &      5.23 &      0.45 &     63.18 \\
		c15900       &      0.93 &      0.64 &      0.99 &      3.19 &      0.89 &     23.45 &      1.32 &    345.01 \\
		c20100       &      0.04 &      0.00 &      0.01 &      0.02 &      0.02 &      0.05 &      0.02 &      0.43 \\
		c20200       &      0.12 &      0.07 &      0.13 &      0.11 &      0.18 &      0.60 &      0.13 &      4.89 \\
		c20400       &      0.44 &      0.14 &      0.39 &      0.74 &      0.44 &      4.46 &      0.58 &     45.20 \\
			c201600      &      3.80 &      1.93 &      3.79 &     13.84 &      3.93 &    107.61 &      4.54 &   1470.70 \\
		c30900       &      1.82 &      0.89 &      1.55 &      3.92 &      1.76 &     27.72 &      2.06 &    299.81 \\
		c40400       &      0.30 &      0.15 &      0.20 &      0.23 &      0.31 &      1.66 &      0.35 &     14.54 \\
		c401600      &      5.15 &      2.81 &      5.55 &     12.98 &      4.94 &     95.79 &      5.54 &   1068.12 \\
		c60900       &      1.80 &      1.00 &      1.77 &      2.74 &      1.88 &     15.66 &      2.10 &    157.51 \\
		c801600      &      6.52 &      3.22 &      6.65 &     10.75 &      6.52 &     72.67 &      7.02 &    710.12 \\
		d05100       &      0.08 &      0.06 &      0.08 &      0.26 &      0.10 &      2.49 &      0.42 &     33.43 \\
		d05200       &      0.50 &      0.42 &      0.48 &      1.64 &      0.55 &     15.32 &      1.44 &    212.24 \\
		d10100       &      0.17 &      0.15 &      0.16 &      0.33 &      0.18 &      2.55 &      0.50 &     26.35 \\
		d10200       &      0.39 &      0.36 &      0.54 &      1.56 &      0.64 &     11.84 &      1.08 &    160.46 \\
		d10400       &      1.02 &      0.93 &      1.18 &      4.85 &      1.33 &     41.06 &      2.28 &    572.74 \\
		d15900       &      7.97 &      5.16 &      7.94 &     33.69 &      8.11 &    276.86 &     11.39 &   3736.91 \\
		d20100       &      0.17 &      0.08 &      0.18 &      0.26 &      0.15 &      1.46 &      0.37 &     13.87 \\
		d20200       &      0.62 &      0.37 &      0.56 &      1.44 &      0.72 &      9.69 &      1.07 &    105.33 \\
		d20400       &      2.07 &      1.19 &      1.73 &      6.03 &      2.28 &     48.39 &      3.16 &    666.20 \\
		d201600      &     23.15 &     16.19 &     22.44 &    109.09 &     22.93 &    956.15 &     28.75 &  13170.76 \\
		d30900       &      7.66 &      5.14 &      7.71 &     28.88 &      8.35 &    225.16 &     10.63 &   3135.95 \\
		d40400       &      1.42 &      0.71 &      1.59 &      2.47 &      1.68 &     22.41 &      2.81 &    269.77 \\
		d401600      &      4.87 &      3.10 &      5.45 &     19.70 &      5.80 &    181.22 &     10.61 &   2783.94 \\
		d60900       &      3.17 &      1.21 &      3.17 &      5.44 &      3.22 &     45.17 &      5.11 &    618.05 \\
		d801600      &      5.47 &      2.96 &      5.23 &     13.27 &      5.76 &    114.97 &      8.65 &   1749.76 \\
		e05100       &      0.00 &      0.01 &      0.01 &      0.06 &      0.02 &      0.16 &      0.05 &      1.82 \\
		e05200       &      0.07 &      0.02 &      0.07 &      0.06 &      0.06 &      0.64 &      0.11 &      6.78 \\
		e10100       &      0.05 &      0.02 &      0.04 &      0.03 &      0.02 &      0.13 &      0.04 &      1.16 \\
		e10200       &      0.03 &      0.03 &      0.05 &      0.11 &      0.08 &      0.36 &      0.10 &      4.33 \\
		e10400       &      0.21 &      0.10 &      0.17 &      0.32 &      0.18 &      2.17 &      0.28 &     22.89 \\
		e15900       &      0.70 &      0.29 &      0.62 &      1.65 &      0.61 &     11.39 &      0.96 &    165.71 \\
		e20100       &      0.06 &      0.03 &      0.04 &      0.02 &      0.05 &      0.06 &      0.09 &      0.98 \\
		e20200       &      0.09 &      0.06 &      0.09 &      0.12 &      0.13 &      0.32 &      0.17 &      3.29 \\
		e20400       &      0.30 &      0.12 &      0.33 &      0.35 &      0.31 &      1.94 &      0.35 &     19.55 \\
		e201600      &      1.90 &      1.00 &      1.96 &      5.87 &      2.03 &     43.33 &      2.64 &    667.57 \\
		e40400       &      0.41 &      0.18 &      0.52 &      0.46 &      0.57 &      1.87 &      0.45 &     16.77 \\
		e401600      &      3.73 &      1.65 &      3.89 &      7.34 &      3.77 &     49.00 &      4.48 &    530.59 \\
		e60900       &      2.21 &      0.84 &      2.24 &      2.29 &      2.59 &     12.27 &      2.56 &    116.10 \\
		e801600      &      5.78 &      2.55 &      6.23 &      7.08 &      6.10 &     43.12 &      6.83 &    425.78 \\
		\hline
\end{longtable}
\end{center}

\begin{center}
	\begin{longtable}{|c|cc|}
		\caption{Performance comparison of the exact lifting with the approximate lifting on the binary knapsack set for benchmark MIPLIB 2017}
		\label{miplib}\\	
		\hline
	Instance  &  Exact lifting  &  Default\\\hline
		\endfirsthead
		\hline
	Instance  &  Exact lifting   &  Default\\\hline
		\endhead
		\hline
		\multicolumn{3}{l}{{Continued on next page}}  
		\endfoot
		\hline
		\endlastfoot
		30n20b8                      & 4702.75 & 7173.74 \\
		app1-2                       & 6016.68 & 5877.89 \\
		beasleyC3                    & 41.60 & 50.96 \\
		binkar10\_1                   & 108.82 & 89.44 \\
		blp-ic98                     & 7200.00 & 6817.97 \\
		bnatt400                     & 4008.11 & 3918.22 \\
		bnatt500                     & 6983.35 & 7200.00 \\
		bppc4-08                     & 3569.54 & 7200.00 \\
		brazil3                      & 7198.26 & 5389.18 \\
		cbs-cta                      & 321.39 & 318.48 \\
		CMS750\_4                     & 7200.00 & 5836.13 \\
		cod105                       & 437.67 & 471.62 \\
		csched007                    & 5974.04 & 5627.76 \\
		csched008                    & 2442.16 & 2521.41 \\
		dano3\_3                      & 438.74 & 387.96 \\
		dano3\_5                      & 1616.21 & 1630.48 \\
		decomp2                      & 102.41 & 122.84 \\
		drayage-100-23               & 5.32 & 7.81 \\
		eil33-2                      & 384.36 & 380.44 \\
		enlight\_hard                 & 63.82 & 37.93 \\
		exp-1-500-5-5                & 20.84 & 25.58 \\
		fastxgemm-n2r6s0t2           & 4972.64 & 4717.03 \\
		fiball                       & 192.56 & 173.20 \\
		gen-ip002                    & 7048.64 & 6381.23 \\
		gen-ip054                    & 3522.04 & 3391.54 \\
		glass4                       & 1884.21 & 3784.45 \\
		gmu-35-40                    & 1212.23 & 1049.56 \\
		gmu-35-50                    & 3193.28 & 3878.76 \\
		graphdraw-domain             & 6086.20 & 5486.43 \\
		h80x6320d                    & 41.74 & 41.81 \\
		icir97\_tension               & 7041.76 & 6388.15 \\
		irp                          & 3544.56 & 2358.12 \\
		istanbul-no-cutoff           & 2400.61 & 2294.86 \\
		lotsize                      & 3303.36 & 3491.45 \\
		mad                          & 5586.01 & 7200.00 \\
		map10                        & 2810.11 & 3327.55 \\
		map16715-04                  & 7114.21 & 7133.73 \\
		markshare\_4\_0                & 591.95 & 492.62 \\
		mas74                        & 5016.43 & 4530.31 \\
		mas76                        & 583.48 & 383.28 \\
		mc11                         & 117.66 & 106.49 \\
		mcsched                      & 1426.70 & 1268.31 \\
		mik-250-20-75-4              & 107.77 & 127.47 \\
		milo-v12-6-r2-40-1           & 5497.19 & 5243.09 \\
		mushroom-best                & 882.71 & 808.92 \\
		mzzv11                       & 2871.43 & 2701.18 \\
		n5-3                         & 60.07 & 57.42 \\
		neos-1122047                 & 32.06 & 32.38 \\
		neos-1171737                 & 6209.80 & 5932.94 \\
		neos-1445765                 & 123.47 & 115.45 \\
		neos-1456979                 & 651.78 & 765.21 \\
		neos17                       & 15.00 & 15.97 \\
		neos-2075418-temuka          & 1020.11 & 1007.24 \\
		neos-2746589-doon            & 4070.47 & 3945.23 \\
		neos-2978193-inde            & 40.42 & 51.14 \\
		neos-3004026-krka            & 3397.61 & 1423.82 \\
		neos-3216931-puriri          & 981.73 & 2141.07 \\
		neos-3381206-awhea           & 4.86 & 5.51 \\
		neos-3402294-bobin           & 287.00 & 368.98 \\
		neos-3627168-kasai           & 140.76 & 89.23 \\
		neos-3988577-wolgan          & 2641.20 & 2693.35 \\
		neos-4338804-snowy           & 6236.77 & 7200.00 \\
		neos-4413714-turia           & 2024.26 & 1980.85 \\
		neos-4647030-tutaki          & 4147.27 & 4703.56 \\
		neos-4722843-widden          & 711.33 & 838.50 \\
		neos-4738912-atrato          & 1312.34 & 1301.97 \\
		neos-5107597-kakapo          & 6717.46 & 7200.00 \\
		neos5                        & 132.03 & 241.35 \\
		neos-662469                  & 4905.79 & 4523.30 \\
		neos-787933                  & 6458.40 & 4621.95 \\
		neos-827175                  & 8.95 & 10.13 \\
		neos-848589                  & 3691.75 & 3692.04 \\
		neos-860300                  & 24.94 & 32.18 \\
		neos-873061                  & 1244.21 & 784.95 \\
		neos8                        & 22.57 & 24.15 \\
		neos-911970                  & 24.25 & 22.84 \\
		neos-933966                  & 385.22 & 415.12 \\
		neos-950242                  & 53.94 & 64.74 \\
		neos-960392                  & 4738.25 & 4626.11 \\
		net12                        & 2002.72 & 1847.78 \\
		nexp-150-20-8-5              & 5333.77 & 5144.23 \\
		ns1116954                    & 1849.68 & 2145.70 \\
		ns1208400                    & 670.20 & 587.24 \\
		ns1644855                    & 6054.63 & 6045.78 \\
		ns1830653                    & 1093.76 & 1379.12 \\
		p200x1188c                   & 6.84 & 7.81 \\
		peg-solitaire-a3             & 6683.16 & 7200.00 \\
		pg5\_34                       & 2049.53 & 2214.81 \\
		pg                           & 41.83 & 25.33 \\
		physiciansched6-2            & 172.96 & 179.71 \\
		piperout-08                  & 182.73 & 172.35 \\
		pk1                          & 1375.60 & 1293.82 \\
		qap10                        & 582.28 & 569.55 \\
		radiationm18-12-05           & 6245.72 & 6697.14 \\
		rail507                      & 2472.43 & 2405.87 \\
		ran14x18-disj-8              & 2962.21 & 3007.67 \\
		rmatr100-p10                 & 396.01 & 414.18 \\
		rocI-4-11                    & 5317.13 & 6486.73 \\
		rocII-5-11                   & 7080.53 & 7200.00 \\
		roi2alpha3n4                 & 6847.99 & 6537.33 \\
		roll3000                     & 108.28 & 122.36 \\
		satellites2-40               & 3862.77 & 4436.94 \\
		satellites2-60-fs            & 3133.13 & 2881.86 \\
		sct2                         & 4499.37 & 4366.58 \\
		seymour1                     & 583.46 & 599.23 \\
		snp-02-004-104               & 611.82 & 632.54 \\
		sp150x300d                   & 0.65 & 0.74 \\
		supportcase18                & 4654.59 & 6431.25 \\
		supportcase26                & 4880.09 & 4091.99 \\
		supportcase33                & 445.64 & 499.20 \\
		supportcase40                & 4191.38 & 4073.73 \\
		supportcase7                 & 305.43 & 298.83 \\
		swath1                       & 57.88 & 53.41 \\
		swath3                       & 489.38 & 588.41 \\
		tbfp-network                 & 3112.30 & 1814.67 \\
		timtab1                      & 191.74 & 201.37 \\
		tr12-30                      & 898.47 & 1009.72 \\
		traininstance6               & 1137.94 & 1208.18 \\
		trento1                      & 7200.00 & 7167.33 \\
		triptim1                     & 2969.15 & 3070.13 \\
		uccase12                     & 51.65 & 47.48 \\
		uccase9                      & 6045.03 & 5681.74 \\
		uct-subprob                  & 4690.73 & 5492.88 \\
		unitcal\_7                    & 416.97 & 340.39 \\
		wachplan                     & 1066.23 & 1612.68 \\
	\end{longtable}
\end{center}
\end{appendices}

%% file: tex/sn-bibliography.bib
@article{kaparis2008local,
	title={Local and global lifted cover inequalities for the 0-1 multidimensional knapsack problem},
	author={Kaparis, Konstantinos and Letchford, Adam N},
	journal={Eur. J. Oper. Res.},
	volume={186},
	number={1},
	pages={91--103},
	year={2008},
	note={\url{https://doi.org/10.1016/j.ejor.2007.01.032}},
	publisher={Elsevier}
}

@article{vasilyev2016implementation,
	title={An implementation of exact knapsack separation},
	author={Vasilyev, Igor and Boccia, Maurizio and Hanafi, Sa{\"\i}d},
	journal={J. Glob. Optim.},
	volume={66},
	number={1},
	pages={127--150},
	year={2016},
	note={\url{https://doi.org/10.1007/s10898-015-0294-3}},
	publisher={Springer}
}

@article{ceselli2009computational,
	title={A computational evaluation of a general branch-and-price framework for capacitated network location problems},
	author={Ceselli, Alberto and Liberatore, Federico and Righini, Giovanni},
	journal={Ann. Oper. Res.},
	volume={167},
	pages={209--251},
	year={2009},
	note={\url{https://doi.org/10.1007/s10479-008-0375-5}}
}

@Inbook{achterberg2013mixed,
	title={Mixed integer programming: Analyzing 12 years of progress},
	author={Achterberg, Tobias and Wunderling, Roland},
	booktitle={Facets of combinatorial optimization: Festschrift for martin gr{\"o}tschel},
	pages={449--481},
	year={2013},
	publisher={Springer},
	address={Berlin, Heidelberg},
	note={\url{https://doi.org/10.1007/978-3-642-38189-8_18}}
}

@article{gomory1969some,
	title={Some polyhedra related to combinatorial problems},
	author={Gomory, Ralph E},
	journal={Linear Algebra Appl. },
	volume={2},
	number={4},
	pages={451--558},
	year={1969},
	note={\url{https://doi.org/10.1016/0024-3795(69)90017-2}},
	publisher={Elsevier}
}

@article{nemhauser1974properties,
	title={Properties of vertex packing and independence system polyhedra},
	author={Nemhauser, George L and Trotter Jr, Leslie E},
	journal={Math.  Program.},
	volume={6},
	number={1},
	pages={48--61},
	year={1974},
	note={\url{https://doi.org/10.1007/BF01580222}},
	publisher={Springer}
}

@article{padberg1973facial,
	title={On the facial structure of set packing polyhedra},
	author={Padberg, Manfred W},
	journal={Math.  Program.},
	volume={5},
	number={1},
	pages={199--215},
	year={1973},
	note={\url{https://doi.org/10.1007/BF01580121}},
	publisher={Springer}
}

@article{savelsbergh1997branch,
	title={A branch-and-price algorithm for the generalized assignment problem},
	author={Savelsbergh, Martin},
	journal={Oper. Res.},
	volume={45},
	number={6},
	pages={831--841},
	year={1997},
	note={\url{https://doi.org/10.1287/opre.45.6.831}},
	publisher={INFORMS}
}

@article{han2013exact,
	title={Exact algorithms for a bandwidth packing problem with queueing delay guarantees},
	author={Han, Jinil and Lee, Kyungsik and Lee, Chungmok and Park, Sungsoo},
	journal={INFORMS J. Comput.},
	volume={25},
	number={3},
	pages={585--596},
	year={2013},
	note={\url{https://doi.org/10.1287/ijoc.1120.0523}},
	publisher={INFORMS}
}

@article{potts1988algorithms,
	title={Algorithms for scheduling a single machine to minimize the weighted number of late jobs},
	author={Potts, Chris N and Van Wassenhove, LN},
	journal={Manag. Sci.},
	volume={34},
	number={7},
	pages={843--858},
	year={1988},
	note={\url{https://doi.org/10.1287/mnsc.34.7.843}},
	publisher={INFORMS}
}

@article{wolsey1975faces,
	title={Faces for a linear inequality in 0--1 variables},
	author={Wolsey, Laurence A},
	journal={Math.  Program.},
	volume={8},
	number={1},
	pages={165--178},
	year={1975},
		note={\url{https://doi.org/10.1007/BF01580441}},
	publisher={Springer}
}

@article{hammer1975facet,
	title={Facet of regular 0--1 polytopes},
	author={Hammer, Peter L and Johnson, Ellis L and Peled, Uri N},
	journal={Math.  Program.},
	volume={8},
	pages={179--206},
	year={1975},
	note={\url{https://doi.org/10.1007/BF01580442}},
	publisher={Springer}
}

@article{de2002facets,
	title={Facets of the complementarity knapsack polytope},
	author={De Farias Jr, Ismael R and Johnson, Ellis L and Nemhauser, George L},
	journal={Math. Oper. Res.},
	volume={27},
	number={1},
	pages={210--226},
	year={2002},
	note={\url{https://doi.org/10.1287/moor.27.1.210.335}},
	publisher={INFORMS}
}

@article{lin2008box,
	title={Box-constrained quadratic programs with fixed charge variables},
	author={Lin, Tin-Chi and Vandenbussche, Dieter},
	journal={J. Glob. Optim.},
	volume={41},
	pages={75--102},
	year={2008},
	note={\url{https://doi.org/10.1007/s10898-007-9167-8}},
	publisher={Springer}
}

@article{vandenbussche2005polyhedral,
	title={A polyhedral study of nonconvex quadratic programs with box constraints},
	author={Vandenbussche, Dieter and Nemhauser, George L},
	journal={Math.  Program.},
	volume={102},
	pages={531--557},
	year={2005},
	note={\url{https://doi.org/10.1007/s10107-004-0549-0}},
	publisher={Springer}
}

@article{ceria1998cutting,
	title={Cutting planes for integer programs with general integer variables},
	author={Ceria, Sebastian and Cordier, C{\'e}cile and Marchand, Hugues and Wolsey, Laurence A},
	journal={Math.  Program.},
	volume={81},
	pages={201--214},
	year={1998},
	note={\url{https://doi.org/10.1007/BF01581105}},
	publisher={Springer}
}

@article{agra2007lifting,
	title={Lifting two-integer knapsack inequalities},
	author={Agra, Agostinho and Constantino, Miguel Fragoso},
	journal={Math.  Program.},
	volume={109},
	number={1},
	pages={115--154},
	year={2007},
	note={\url{https://doi.org/10.1007/s10107-006-0705-9}}
}

@article{wolsey1976facets,
	title={Facets and strong valid inequalities for integer programs},
	author={Wolsey, Laurence A},
	journal={Oper. Res.},
	volume={24},
	number={2},
	pages={367--372},
	year={1976},
		note={\url{https://doi.org/10.1287/opre.24.2.367}},
	publisher={INFORMS}
}

@article{richard2009valid,
	title={Valid inequalities for {MIP}s and group polyhedra from approximate liftings},
	author={Richard, Jean-Philippe P and Li, Yanjun and Miller, Lisa A},
	journal={Math.  Program.},
	volume={118},
	pages={253--277},
	year={2009},
	note={\url{https://doi.org/10.1007/s10107-007-0190-9}},
	publisher={Springer}
}

@article{basu2019nonunique,
	title={Nonunique lifting of integer variables in minimal inequalities},
	author={Basu, Amitabh and Dey, Santanu S and Paat, Joseph},
	journal={SIAM J. Discrete Math.},
	volume={33},
	number={2},
	pages={755--783},
	year={2019},
	note={\url{https://doi.org/10.1137/17M1117070}},
	publisher={SIAM}
}

@article{basu2012unique,
	title={Unique minimal liftings for simplicial polytopes},
	author={Basu, Amitabh and Cornu{\'e}jols, G{\'e}rard and K{\"o}ppe, Matthias},
	journal={Math. Oper. Res.},
	volume={37},
	number={2},
	pages={346--355},
	year={2012},
	note={\url{https://doi.org/10.1287/moor.1110.0536}},
	publisher={INFORMS}
}

@article{dey2010constrained,
	title={Constrained infinite group relaxations of {MIP}s},
	author={Dey, Santanu S and Wolsey, Laurence A},
	journal={SIAM J. Optim.},
	volume={20},
	number={6},
	pages={2890--2912},
	year={2010},
	note={\url{https://doi.org/10.1137/090754388}},
	publisher={SIAM}
}

@article{hartvigsen1992complexity,
	title={The complexity of lifted inequalities for the knapsack problem},
	author={Hartvigsen, David and Zemel, Eitan},
	journal={Discrete Appl. Math.},
	volume={39},
	number={2},
	pages={113--123},
	year={1992},
	note={\url{https://doi.org/10.1016/0166-218X(92)90162-4}},
	publisher={Elsevier}
}

@article{padberg1975note,
	title={A note on zero-one programming},
	author={Padberg, Manfred W},
	journal={Oper.  Res.},
	volume={23},
	number={4},
	pages={833--837},
	year={1975},
	note={\url{https://doi.org/10.1287/opre.23.4.833}},
	publisher={INFORMS}
}

@article{wolsey1977valid,
	title={Valid inequalities and superadditivity for 0-1 integer programs},
	author={Wolsey, Laurence A},
	journal={Math. Oper. Res.},
	volume={2},
	number={1},
	pages={66--77},
	year={1977},
	note={\url{https://doi.org/10.1287/moor.2.1.66}},
	publisher={INFORMS}
}

@article{crowder1983solving,
	title={Solving large-scale zero-one linear programming problems},
	author={Crowder, Harlan and Johnson, Ellis L and Padberg, Manfred},
	journal={Oper.  Res.},
	volume={31},
	number={5},
	pages={803--834},
	year={1983},
	note={\url{https://doi.org/10.1287/opre.31.5.803}},
	publisher={INFORMS}
}

@article{letchford2019lifted,
	title={On lifted cover inequalities: A new lifting procedure with unusual properties},
	author={Letchford, Adam N and Souli, Georgia},
	journal={Oper. Res. Lett.},
	volume={47},
	number={2},
	pages={83--87},
	year={2019},
	note={\url{https://doi.org/10.1016/j.orl.2018.12.005}},
	publisher={Elsevier}
}

@article{balas1975facets,
	title={Facets of the knapsack polytope},
	author={Balas, Egon},
	journal={Math.  Program.},
	volume={8},
	pages={146--164},
	year={1975},
	note={\url{https://doi.org/10.1007/BF01580440}}
}

@article{balas1978facets,
	title={Facets of the knapsack polytope from minimal covers},
	author={Balas, Egon and Zemel, Eitan},
	journal={SIAM J. Appl. Math.},
	volume={34},
	number={1},
	pages={119--148},
	year={1978},
	note={\url{https://doi.org/10.1137/0134010}},
	publisher={SIAM}
}

@article{gu2000sequence,
	title={Sequence independent lifting in mixed integer programming},
	author={Gu, Zonghao and Nemhauser, George L and Savelsbergh, Martin WP},
	journal={J. Comb. Optim.},
	volume={4},
	pages={109--129},
	year={2000},
	note={\url{https://doi.org/10.1023/A:1009841107478}},
	publisher={Springer}
}

@article{kaparis2010separation,
	title={Separation algorithms for 0-1 knapsack polytopes},
	author={Kaparis, Konstantinos and Letchford, Adam N},
	journal={Math.  Program.},
	volume={124},
	pages={69--91},
	year={2010},
	note={\url{https://doi.org/10.1007/s10107-010-0359-5}},
	publisher={Springer}
}

@article{zemel1989easily,
	title={Easily computable facets of the knapsack polytope},
	author={Zemel, Eitan},
	journal={Math. Oper. Res.},
	volume={14},
	number={4},
	pages={760--764},
	year={1989},
	note={\url{https://doi.org/10.1287/moor.14.4.760}},
	publisher={INFORMS}
}

@article{zeng2011polyhedral,
	title={A polyhedral study on 0-1 knapsack problems with disjoint cardinality constraints: Facet-defining inequalities by sequential lifting},
	author={Zeng, Bo and Richard, Jean-Philippe P},
	journal={Discrete Optim.},
	volume={8},
	number={2},
	pages={277--301},
	year={2011},
	note={\url{https://doi.org/10.1016/j.disopt.2010.09.005}},
	publisher={Elsevier}
}

@article{zeng2011polyhedral2,
	title={A polyhedral study on 0-1 knapsack problems with disjoint cardinality constraints: Strong valid inequalities by sequence-independent lifting},
	author={Zeng, Bo and Richard, Jean-Philippe P},
	journal={Discrete Optim.},
	volume={8},
	number={2},
	pages={259--276},
	year={2011},
	note={\url{https://doi.org/10.1016/j.disopt.2010.09.004}},
	publisher={Elsevier}
}

@article{van2002polyhedral,
	title={Polyhedral results for the edge capacity polytope},
	author={van Hoesel, Stan PM and Koster, Arie MCA and van de Leensel, Robert LMJ and Savelsbergh, Martin WP},
	journal={Math.  Program.},
	volume={92},
	pages={335--358},
	year={2002},
	note={\url{https://doi.org/10.1007/s101070200292}},
	publisher={Springer}
}

@article{zemel1978lifting,
	title={Lifting the facets of zero-one polytopes},
	author={Zemel, Eitan},
	journal={Math.  Program.},
	volume={15},
	pages={268--277},
	year={1978},
	note={\url{https://doi.org/10.1007/BF01609032}},
	publisher={Springer}
}

@article{easton2008simultaneously,
	title={Simultaneously lifting sets of binary variables into cover inequalities for knapsack polytopes},
	author={Easton, Todd and Hooker, Kevin},
	journal={Discrete Optim.},
	volume={5},
	number={2},
	pages={254--261},
	year={2008},
	note={\url{https://doi.org/10.1016/j.disopt.2007.05.003}},
	publisher={Elsevier}
}

@article{weismantel19970,
	title={On the 0/1 knapsack polytope},
	author={Weismantel, Robert},
	journal={Math.  Program.},
	volume={77},
	pages={49--68},
	year={1997},
	note={\url{https://doi.org/10.1007/BF02614517}},
	publisher={Springer}
}

@article{ross1975branch,
	title={A branch and bound algorithm for the generalized assignment problem},
	author={Ross, G Terry and Soland, Richard M},
	journal={Math.  Program.},
	volume={8},
	number={1},
	pages={91--103},
	year={1975},
	note={\url{https://doi.org/10.1007/BF01580430}},
	publisher={Springer}
}

@article{chen2021complexity,
	title={On the complexity of sequentially lifting cover inequalities for the knapsack polytope},
	author={Chen, Wei-Kun and Dai, Yu-Hong},
	journal={Sci. China Math. },
	volume={64},
	pages={211--220},
	year={2021},
	note={\url{https://doi.org/10.1007/s11425-019-9538-1}}
}

@book{wolsey1999integer,
	title={Integer and combinatorial optimization},
	author={Wolsey, Laurence A and Nemhauser, George L},
	year={1999},
	address		= "New {Y}ork",
	publisher={John Wiley \& Sons}
}

@article{ferreira1996solving,
	title={Solving multiple knapsack problems by cutting planes},
	author={Ferreira, Carlos E and Martin, Alexander and Weismantel, Robert},
	journal={SIAM J. Optim.},
	volume={6},
	number={3},
	pages={858--877},
	year={1996},
	note={\url{https://doi.org/10.1137/S1052623493254455}},
	publisher={SIAM}
}

@phdthesis{gu1994lifted,
	title={Lifted cover inequalities for 0-1 and mixed 0-1 integer programs},
	author={Gu, Zonghao},
	year={1995},
	school={Georgia Institute of Technology}
}

@article{gu1998lifted,
	title={Lifted cover inequalities for 0-1 integer programs: Computation},
	author={Gu, Zonghao and Nemhauser, George L and Savelsbergh, Martin WP},
	journal={INFORMS J. Comput.},
	volume={10},
	number={4},
	pages={427--437},
	year={1998},
	note={\url{https://doi.org/10.1287/ijoc.10.4.427}},
	publisher={INFORMS}
}

@article{gu1999lifted,
	title={Lifted cover inequalities for 0-1 integer programs: Complexity},
	author={Gu, Zonghao and Nemhauser, George L and Savelsbergh, Martin WP},
	journal={INFORMS J. Comput.},
	volume={11},
	number={1},
	pages={117--123},
	year={1999},
	note={\url{https://doi.org/10.1287/ijoc.11.1.117}},
	publisher={INFORMS}
}

@phdthesis{kubik2009simultaneously,
	title={Simultaneously lifting multiple sets in binary knapsack integer programs},
	author={Kubik, Lauren Ashley},
	year={2009},
	school={Kansas State University}
}

@article{shebalov2006sequence,
	title={Sequence independent lifting for mixed integer programs with variable upper bounds},
	author={Shebalov, Sergey and Klabjan, Diego},
	journal={Math.  Program.},
	volume={105},
	number={2},
	pages={523--561},
	year={2006},
	note={\url{https://doi.org/10.1007/s10107-005-0664-6}},
	publisher={Springer}
}

@article{bellman1957comment,
	title={Comment on Dantzig's paper on discrete variable extremum problems},
	author={Bellman, Richard},
	journal={Oper.  Res.},
	volume={5},
	number={5},
	pages={723--724},
	year={1957},
	note={\url{https://doi.org/10.1287/opre.5.5.723}},
	publisher={INFORMS}
}

@article{atamturk2004sequence,
	title={Sequence independent lifting for mixed-integer programming},
	author={Atamt{\"u}rk, Alper},
	journal={Oper.  Res.},
	volume={52},
	number={3},
	pages={487--490},
	year={2004},
	note={\url{https://doi.org/10.1287/opre.1030.0099}},
	publisher={INFORMS}
}

@Inbook{kellerer2004multidimensional,
	author="Kellerer, Hans
	and Pferschy, Ulrich
	and Pisinger, David",
	title="Multidimensional Knapsack Problems",
	bookTitle="Knapsack Problems",
	year="2004",
	publisher="Springer",
	address="Berlin, Heidelberg",
	pages="235--283",
	note={\url{https://doi.org/10.1007/978-3-540-24777-7_9}},
}

@article{atamturk2003facets,
	title={On the facets of the mixed--integer knapsack polyhedron},
	author={Atamt{\"u}rk, Alper},
	journal={Math.  Program.},
	volume={98},
	number={1},
	pages={145--175},
	year={2003},
	note={\url{https://doi.org/10.1007/s10107-003-0400-z}},
	publisher={Springer}
}

@article{narisetty2011lifted,
	title={Lifted tableaux inequalities for 0-1 mixed-integer programs: A computational study},
	author={Narisetty, Amar K and Richard, Jean-Philippe P and Nemhauser, George L},
	journal={INFORMS J. Comput.},
	volume={23},
	number={3},
	pages={416--424},
	year={2011},
	note={\url{https://doi.org/10.1287/ijoc.1100.0413}},
	publisher={INFORMS}
}

@article{gu1999lifted2,
	title={Lifted flow cover inequalities for mixed 0-1 integer programs},
	author={Gu, Zonghao and Nemhauser, George L and Savelsbergh, Martin WP},
	journal={Math.  Program.},
	volume={85},
	pages={439--467},
	year={1999},
	note={\url{https://doi.org/10.1007/s101070050067}},
	publisher={Springer}
}

@article{richard2002lifted,
	title={Lifted inequalities for 0-1 mixed integer programming: Basic theory and algorithms},
	author={Richard, J-PP and de Farias Jr, Ismael R and Nemhauser, George L},
	journal={Math.  Program.},
	volume={98},
	number={1},
	pages={89--113},
	year={2003},
	note={\url{https://doi.org/10.1007/s10107-003-0398-2}},
	publisher={Springer}
}

@article{richard2003lifted,
	title={Lifted inequalities for 0-1 mixed integer programming: Superlinear lifting},
	author={Richard, J-PP and de Farias Jr, Ismael R and Nemhauser, George L},
	journal={Math.  Program.},
	volume={98},
	number={1},
	pages={115--143},
	year={2003},
	note={\url{https://doi.org/10.1007/s10107-003-0399-1}},
	publisher={Springer}
}

@article{richard2010lifting,
	title={Lifting inequalities: A framework for generating strong cuts for nonlinear programs},
	author={Richard, Jean-Philippe P and Tawarmalani, Mohit},
	journal={Math.  Program.},
	volume={121},
	pages={61--104},
	year={2010},
	note={\url{https://doi.org/10.1007/s10107-008-0226-9}},
	publisher={Springer}
}

@article{nguyen2018deriving,
	title={Deriving convex hulls through lifting and projection},
	author={Nguyen, Trang T and Richard, Jean-Philippe P and Tawarmalani, Mohit},
	journal={Math.  Program.},
	volume={169},
	pages={377--415},
	year={2018},
	note={\url{https://doi.org/10.1007/s10107-017-1138-3}},
	publisher={Springer}
}

@article{chung2014lifted,
	title={Lifted inequalities for mixed-integer bilinear covering sets},
	author={Chung, Kwanghun and Richard, Jean-Philippe P and Tawarmalani, Mohit},
	journal={Math.  Program.},
	volume={145},
	number={1},
	pages={403--450},
	year={2014},
	note={\url{https://doi.org/10.1007/s10107-013-0652-1}},
	publisher={Springer}
}

@phdthesis{gupte2012mixed,
	title={Mixed integer bilinear programming with applications to the pooling problem},
	author={Gupte, Akshay},
	year={2012},
	schppl={Georgia Institute of Technology}
}

@article{gu2023lifting,
	title={Lifting convex inequalities for bipartite bilinear programs},
	author={Gu, Xiaoyi and Dey, Santanu S and Richard, Jean-Philippe P},
	journal={Math.  Program.},
	volume={197},
	number={2},
	pages={587--619},
	year={2023},
	note={\url{https://doi.org/10.1007/s10107-021-01759-3}},
	publisher={Springer}
}

@article{fawzi2022lifting,
	title={Lifting for simplicity: Concise descriptions of convex sets},
	author={Fawzi, Hamza and Gouveia, Joao and Parrilo, Pablo A and Saunderson, James and Thomas, Rekha R},
	journal={SIAM Rev.},
	volume={64},
	number={4},
	pages={866--918},
	year={2022},
	note={\url{https://doi.org/10.1137/20M1324417}},
	publisher={SIAM}
}

@article{gomez2018submodularity,
	title={Submodularity and valid inequalities in nonlinear optimization with indicator variables},
	author={G{\'o}mez, A},
	journal={Available at Optimization online},
	year={2018}
}

@article{dey2010composite,
	title={Composite lifting of group inequalities and an application to two-row mixing inequalities},
	author={Dey, Santanu S and Wolsey, Laurence A},
	journal={Discrete Optim.},
	volume={7},
	number={4},
	pages={256--268},
	year={2010},
	note={\url{https://doi.org/10.1016/j.disopt.2010.06.001}},
	publisher={Elsevier}
}

@article{atamturk2011lifting,
	title={Lifting for conic mixed-integer programming},
	author={Atamt{\"u}rk, Alper and Narayanan, Vishnu},
	journal={Math.  Program.},
	volume={126},
	number={2},
	pages={351--363},
	year={2011},
	note={\url{https://doi.org/10.1007/s10107-009-0282-9}},
	publisher={Springer}
}

@article{marchand19990,
	title={The 0-1 knapsack problem with a single continuous variable},
	author={Marchand, Hugues and Wolsey, Laurence A},
	journal={Math.  Program.},
	volume={85},
	pages={15--33},
	year={1999},
	note={\url{https://doi.org/10.1007/s101070050044}},
	publisher={Springer}
}

@phdthesis{martin1999integer,
	title={Integer programs with block structure},
	author={Martin, Alexander},
	year={1999},
	school={TU Berlin}
}

@article{hoffman1991improving,
	title={Improving {LP}-representations of zero-one linear programs for branch-and-cut},
	author={Hoffman, Karla L and Padberg, Manfred},
	journal={ORSA J. Comput.},
	volume={3},
	number={2},
	pages={121--134},
	year={1991},
	note={\url{https://doi.org/10.1287/ijoc.3.2.121}},
	publisher={INFORMS}
}

@article{huangfu2018parallelizing,
	title={Parallelizing the dual revised simplex method},
	author={Huangfu, Qi and Hall, JA Julian},
	journal={Math.  Program.  Comput.},
	volume={10},
	number={1},
	pages={119--142},
	year={2018},
	note={\url{https://doi.org/10.1007/s12532-017-0130-5}},
	publisher={Springer}
}

@misc{cbc,
	author		= "Forrest, John J. and Hafer, Lou and Goncalves, Joao P. and
	Santos, Haroldo G. and	Brito, Samuel S.",
	title			= " coin-or/cbc: Release releases/2.10.10", 
	year			= "2023",
	note			= "\url{https://doi.org/10.5281/zenodo.7843975}"
}

@misc{bolusani2024scip,
	title={The {SCIP} Optimization Suite 9.0}, 
	author={Suresh Bolusani and Mathieu Besanon and Ksenia Bestuzheva and Antonia Chmiela and Jono Mixed integer bilinear programming withion\'{i}sio and Tim Donkiewicz and Jasper van Doornmalen and Leon Eifler and Mohammed Ghannam and Ambros Gleixner and Christoph Graczyk and Katrin Halbig and Ivo Hedtke and Alexander Hoen and Christopher Hojny and Rolf van der Hulst and Dominik Kamp and Thorsten Koch and Kevin Kofler and Jurgen Lentz and Julian Manns and Gioni Mexi and Erik M\"{u}hmer and Marc E. Pfetsch and Franziska Schl\"{o}sser and Felipe Serrano and Yuji Shinano and Mark Turner and Stefan Vigerske and Dieter Weninger and Liding Xu},
	year={2024},
	eprint={2402.17702},
	archivePrefix={arXiv},
	primaryClass={math.OC},
	url={https://arxiv.org/abs/2402.17702}
}
